\documentclass[12pt,reqno]{amsart}
\numberwithin{equation}{section}
\numberwithin{table}{section}
\usepackage{amssymb,amsmath,mathrsfs}
\usepackage{epsfig}
\usepackage{graphicx}
\usepackage{fullpage}
\usepackage{hyperref}

\usepackage[OT2,T1]{fontenc}
\DeclareSymbolFont{cyrletters}{OT2}{wncyr}{m}{n}
\DeclareMathSymbol{\Sha}{\mathalpha}{cyrletters}{"58}
\DeclareMathSymbol{\Be}{\mathalpha}{cyrletters}{"42}

\input xy
\xyoption{all}

\theoremstyle{plain}
\newtheorem{thm}{Theorem}[section]
\newtheorem{theo}{Theorem}

\newtheorem{lem}[thm]{Lemma}
\newtheorem{lemma}[thm]{Lemma}
\newtheorem{coro}[thm]{Corollary}

\newtheorem{prop}[thm]{Proposition}

\newtheorem{conj}[thm]{Conjecture}

\theoremstyle{definition} \theoremstyle{definition}
\newtheorem{defn}[thm]{Definition}

\newtheorem{exam}[thm]{Example}

\newtheorem{rem}[thm]{Remark}           

\theoremstyle{remark}

\newcommand{\al}{\alpha}
\newcommand{\C}{\mathbb C}
\newcommand{\Q}{\mathbb Q}
\newcommand{\Z}{\mathbb Z}
\newcommand{\R}{\mathbb R}
\newcommand{\Adele}{\mathbb A}

\newcommand\Gm{\mathbb{G}_\mathrm{m}}

\def\cD{{D}}

\def\cT{{\mathcal T}}

\def\bDelta{{\boldsymbol{\Delta}}}

\def\sK{{\mathsf K}}

\def\y{{\mathbf y}}

\def\G{{\mathsf G}}
\def\H{{\mathsf H}}

\def\T{{\mathsf T}}
\def\S{{\mathsf S}}
\def\B{{\mathsf B}}

\newcommand{\Hom}{{\rm Hom}}

\newcommand{\K}{{\mathsf K}}

\let\Re\relax

\DeclareMathOperator{\Br}{Br}
\DeclareMathOperator{\Re}{Re}

\def\id{{\rm id}}

\def\ra{\rightarrow}

\def\A{{\mathbb A}}
\def\C{{\mathbb C}}
\def\F{{\mathbb F}}

\def\Q{{\mathbb Q}}

\def\Z{{\mathbb Z}}
\def\C{{\mathbb C}}
\def\N{{\mathbb N}}

\DeclareMathOperator{\Pic}{Pic} 
\DeclareMathOperator{\Gal}{Gal} 
 
\DeclareMathOperator{\Spec}{Spec}

\DeclareMathOperator{\vol}{vol}

\DeclareMathOperator{\Eff}{Eff}

\DeclareMathOperator{\inv}{inv} 
\DeclareMathOperator{\res}{\partial}

\DeclareMathOperator{\sign}{sign}

\def\F{{\Bbb F}}

\def\eps{{\epsilon}}

\def\bfs{{\mathbf s}}

\def\sK{{\mathsf K}}

\usepackage[usenames,dvipsnames]{color}

\newcommand{\sho}[1]{{\color{blue} \sf $\clubsuit\clubsuit\clubsuit$ Sho: [#1]}}

\newcommand{\dan}[1]{{\color{red} \sf $\clubsuit\clubsuit\clubsuit$ Dan: [#1]}}

\author{Dylon Chow}
\address{University of Washington \\ 
Department of Mathematics \\ 
Box 354350, C-138 Padelford* \\
Seattle, WA 98195-4350}
\email{chow.dylon@gmail.com}
\urladdr{https://dchow4.wixsite.com/math}

\author{Daniel Loughran} 
  \address{
  Department of Mathematical Sciences \\
University of Bath \\
Claverton Down \\
Bath \\
BA2 7AY \\
UK}
\urladdr{https://sites.google.com/site/danielloughran}

\author{Ramin Takloo-Bighash}
\address{Department of Mathematics, Statistics, and Computer Science, University of Illinois at Chicago, 851 S. Morgan St (M/C 251), Chicago, IL 60607}
\email{rtakloo@uic.edu}
\urladdr{https://sites.google.com/site/rtakloo}

\author{Sho Tanimoto}
\address{Graduate School of Mathematics, Nagoya University, Furo-cho, Chikusa-ku, Nagoya, 464-8602, Japan}
\email{sho.tanimoto@math.nagoya-u.ac.jp}
\urladdr{https://shotanimoto.wordpress.com/}

\begin{document}

\title{Campana points on wonderful compactifications}


\begin{abstract}
We prove a variant of Manin's conjecture for Campana points on wonderful compactifications of semi-simple algebraic groups of adjoint type. We use this to provide evidence for a new conjecture on the leading constant in Manin's conjecture for Campana points.
\end{abstract}

\date{\today}

\maketitle

\tableofcontents

\section{Introduction}
\label{sect:intro}

Manin's conjecture concerns an asymptotic formula for the counting function of rational points on a rationally connected smooth projective variety over a number field,  developed in a series of papers \cite{FMT,batyrev-m90,Pey95,batyrev-t98,Pey03,Pey17,LST18}. A fertile testing ground  is the class of homogenous spaces, and one of the prominent methods to study this class of varieties is the height zeta function method, utilised in numerous works \cite{FMT,BT,chambert-t02,STT,TT12,ST16}.
The same method has also been applied to the case of integral points in \cite{C-T-Ga,C-T-Torus,TBT13,Chow19,Santens}. In particular \cite{Santens} proposes a version of log Manin's conjecture for integral points.

In \cite{Cam05,Abra09,Cam15}, Campana and Abramovich introduced the notion of \textit{Campana points} which interpolate between rational points and integral points.
The papers \cite{PSTVA19,BY19} started a systematic study of Manin's conjecture for klt Campana points, and a version of Manin's conjecture for klt Campana points was put forward in \cite{PSTVA19} and proved for equivariant compactifications of vector groups using the height zeta function method. These papers triggered extensive activities in the field, attested by \cite{PS20,Xiao21,Str22,Shute21a,Shute21b,BBKOPW}, and currently a log Manin's conjecture for Campana points is an active area of research in mathematics. In this paper, we add another contribution to this literature and establish log Manin's conjecture for Campana points on wonderful compactifications of semi-simple groups of adjoint type; our approach is based on the height zeta function method for semi-simple groups developed in the series of works \cite{STT,TBT13,LTBT18,Chow19}. 

More than this however, the original conjectural constant put forward in \cite{PSTVA19} was shown to be incorrect in \cite{Str22} and \cite{Shute21b}. We suggest a corrected version of this leading constant (Conjecture \ref{conj:leading_constant}), which we verify to hold in our case. This is given as a (possibly infinite) sum over integrals of local invariants of Brauer group elements with respect to a suitable Tamagawa measure.

\subsection{The main result}\label{sect:mainresult}

We now state our main theorem more precisely. Let $F$ be a number field, $\G$ a semi-simple adjoint group over $F$ of rank $\geq 2$, and $X$ the wonderful compactification of $\G$ (see \cite{concini-p83} for its construction). The boundary divisor $D=X \backslash \G$ is a strict normal crossings divisor on $X$. Let $(D_\alpha)_{\alpha \in \mathcal{A}}$ be its irreducible components defined over $F$.  Let $S$ be a finite set of places of $F$ containing all archimedean places such that there is a good integral model $(\mathcal{X},\mathcal{D})$ for $(X,D)$ over the ring of $S$-integers $\mathcal{O}_{F,S}$ of $F$ in the sense of \cite[Section 3.2]{PSTVA19}. For each $\alpha \in \mathcal{A},$ let $\epsilon_\alpha$ be an element of
\[
\{1-1/m: m \in \mathbb{Z}_{\geq 1}\} \cup \{1\}.
\]
Let $\epsilon=(\epsilon_\alpha)_{\alpha \in \mathcal{A}}$, and set 
\[
D_\epsilon=\sum_{\alpha \in \mathcal{A}} \epsilon_\alpha D_\alpha, \ \ \mathcal{D}_\epsilon=\sum_{\alpha \in \mathcal{A}}\epsilon_\alpha \mathcal{D}_\alpha,
\]
where $\mathcal{D}_\alpha$ is the Zariski closure of $D_\alpha$ in $\mathcal{X}$. The choice of divisor $D$, weight vector $\epsilon$, and integral model $(\mathcal{X},\mathcal{D}_\epsilon)$ determines a set $(\mathcal{X},\mathcal{D}_\epsilon)(\mathcal{O}_{F,S})$ of Campana $\mathcal{O}_{F,S}$-points on $(\mathcal{X},\mathcal{D}_\epsilon)$ (see Definition~\ref{defn:Campana_points} for details).

Let $\mathcal{L}=(L,||.||)$ be a pair consisting of a big line bundle $L$ on $X$ and a smooth adelic metrization on $L$. Let $\H_\mathcal{L}:X(F) \rightarrow \mathbb{R}_{>0}$ be the height function determined by $\mathcal{L}$. 
Let $\G(F)_\epsilon$ be the set of $F$-rational points in $\G$ which extend to Campana $\mathcal{O}_{F,S}$-points on $(\mathcal{X},\mathcal{D}_\epsilon)$, i.e.,
$$\G(F)_\epsilon=\G(F) \cap (\mathcal{X},\mathcal{D}_\epsilon)(\mathcal{O}_{F,S}).$$
We will obtain an asymptotic formula, as $B \rightarrow \infty$, for
\[
\mathsf{N}(\G(F)_\epsilon,\mathcal{L},B) = \# \{    
\gamma \in \G(F)_\epsilon : \mathsf H_\mathcal{L}(\gamma)\le B\}
\]
in terms of the invariants $a=a((X, D_\epsilon), L),$ $b=b(F, (X,D_\epsilon),L),$ and $c=c(F,S,(\mathcal{X},\mathcal{D}_\epsilon),\mathcal{L})$, which are defined in Section~\ref{subsec:logManin} as well as Conjecture~\ref{conj:leading_constant} when $(X, D_\epsilon)$ is klt. The main result of this article is the following:

\begin{theo} 
\label{theo:maintheorem} 
	Assume that the quasi-split inner form of $\G$ has rank $\geq 2$ and $(X, D_\epsilon)$ is klt, i.e., $\epsilon_\alpha < 1$ for any $\alpha \in \mathcal A$.
	Then we have
	\begin{equation}
	\label{eq:asymptotic_conjecture}
	\mathsf{N}(\G(F)_\epsilon,\mathcal{L},B) = c B^a \log(B)^{b-1}(1+o(1)),  \quad B\ra \infty.
	\end{equation}
\end{theo}

We expect Theorem \ref{theo:maintheorem} to be true for the rank $1$ case, but we will not include a proof here,  c.f., Remark \ref{rem:rank1}.

\subsection{The leading constant}
We obtain an explicit expression for the leading constant appearing in Theorem \ref{theo:maintheorem}. A conjecture for the leading constant for Campana points was put forward in \cite[Section~3.3]{PSTVA19}. This conjecture was shown to be incorrect independently by Streeter \cite{Str22} and Shute \cite{Shute21b}. The conjectural constant closely mirrored Peyre's constant \cite{Pey95} from Manin's conjecture, and the problematic part concerned the factor $\#\mathrm{H}^1(\Gamma_F, \Pic \bar{U})$ proposed as a replacement for the $\beta$ constant (this part of the constant first appeared in \cite{BT2}).

We suggest a fix (Conjecture \ref{conj:leading_constant}) which is compatible with our expression. The fix is somewhat subtle as it involves taking a sum over elements of a \emph{Campana Brauer group} and taking integrals of the corresponding local invariants. See Theorem \ref{thm:leading_constant} and Section \ref{sec:leadingconstant} for further details.

\subsection{Methodology}
Let us explain an idea of the height zeta function method developed in a series of papers, e.g., \cite{FMT,BT,chambert-t02,STT,TT12,ST16} for rational points, \cite{C-T-Ga,TBT13,Chow19,Santens} for integral points, and \cite{PSTVA19,Str22} for Campana points. Our adelically metrized big line bundle $\mathcal L$ induces an adelic height function
\[
\mathsf H_{\mathcal L} : \G(\A_F) \to \mathbb R_{>0},
\]
where $\A_F$ is the ring of adeles of $F$. We have the notion of adelic Campana set
\[
\G(\A_F)_\epsilon := \G(\A_F)\cap \left(\prod_v(\mathcal{X},\mathcal{D}_\epsilon)(\mathcal{O}_v)\right).
\]
We denote by $\delta_\epsilon : \G(\A_F) \to \mathbb R_{>0}$ the characteristic function of $\G(\A_F)_\epsilon$.
Then we consider the following Dirichlet series
\[
\mathsf Z(s, g) := \sum_{\gamma \in \G(F)} \mathsf H_{\mathcal L}(\gamma g)^{-s}\delta_\epsilon (\gamma g), \quad s \in \C, \quad g \in \G(\A_F).
\]
When $\Re(s)$ is sufficiently large, this Dirichlet series converges to a holomorphic function in $s$ and a continuous and bounded function in $g \in \G(\A_F)$.  Moreover this function is left invariant under $\G(F)$ so that it descends to a function
\[
\mathsf Z(s, g) \in \mathsf L^2(\G(F)\backslash \G(\A_F)).
\]
By a Tauberian theorem, our goal here is to obtain a meromorphic continuation of this function in $s$ and identify the rightmost pole, its order, and the leading constant. To this end we employ the spectral decomposition of $\mathsf L^2(\G(F)\backslash \G(\A_F))$ 
\[
\mathsf L^2(\G(F)\backslash \G(\A_F)) = \widehat{\bigoplus}_\pi \mathsf H_{\pi}
\]
into irreducible representations of Hilbert spaces. Here the r\^{o}le of an orthogonal basis is played by automorphic forms, i.e., cusp forms and Eisenstein series, and this spectral theory has been developed originally in \cite{Arthur1,Arthur3,Arthur2}, and subsequently in \cite{STT,TBT13,LTBT18,Chow19} in the context of height zeta function method for semi-simple groups of adjoint type.

\subsection{Structure of the paper}
In Section~\ref{sect:notions}, we recall basic notation, conventions, and preliminaries. In Section~\ref{sec:Campana}, we recall the notion of wonderful compactifications and Campana points. Then in Section~\ref{sec:heights}, we turn to the notion of height functions and height zeta functions and we discuss their basic properties which are needed in our proofs. Section~\ref{sec:heightI} is devoted to the analysis of one-dimensional representations in the spectral decomposition of the height zeta function and Section~\ref{sec:heightII} is devoted to the analysis of infinite dimensional representations. In Section~\ref{sect:poles}, we analyze the rightmost pole and its order of the height zeta function using estimates established in previous sections. In Section~\ref{sec:leadingconstant}, we formulate our conjectural leading constant for log Manin's conjecture for klt Campana points and verify it in our situation. 

\bigskip

\noindent
{\bf Acknowledgements:}
We are grateful to Tim Santens and Alec Shute for discussions on the leading constant.
We thank Marta Pieropan and Boaz Moerman for comments on this paper.
Daniel Loughran was supported by UKRI Future Leaders Fellowship \texttt{MR/V021362/1}.
Ramin Takloo-Bighash was partially supported by a Collaboration Grant from the Simons Foundation. Sho Tanimoto was partially supported by JST FOREST program Grant number JPMJFR212Z, by JSPS KAKENHI Grand-in-Aid (B) 23K25764, by JSPS Bilateral Joint Research Projects Grant number JPJSBP120219935, by Inamori Foundation, by JSPS KAKENHI Early-Career Scientists Grant number 19K14512, and by MEXT Japan, Leading Initiative for Excellent Young Researchers (LEADER). %

\section{Notation, conventions, and preliminaries}
\label{sect:notions}

In this section, we recall or fix some notation and conventions, often to be used without reference.

\subsection{Number fields} Throughout this paper $F$ is a field of characteristic $0$. When $F$ is a number field, 
$\mathcal{O}_F$ is the ring of integers of $F$, and $\mathbb{A}_F$ is the $F$-algebra of adeles of $F$ endowed with the usual locally compact topology. Let $\Omega_F$ be the set of places of $F$, and $\Omega_F^{<\infty}$ (resp. $\Omega_F^\infty$) the set of all nonarchimedean (resp. archimedean) places. For $v \in \Omega_F$, $F_v$ denotes the completion of $F$ at $v$, and $|.|_v$ the normalized absolute value on $F_v$. (See \cite[Section 2.1]{PSTVA19} for a way to normalize absolute values.) For non-archimedean $v$, $\mathcal{O}_v$ denotes the ring of integers of $F_v$, $\mathfrak{m}_v$ the maximal ideal of $\mathcal{O}_v$, $k_v$ the residue field, $q_v$ the order of $k_v$, and $\varpi_v$ a uniformizer. We write $\A_S$ for $\prod_{v\notin S}' F_v$, and $\A_f$ for $\A_{\Omega^{\infty}_F}$. For any finite set $S \subset \Omega_F$ containing $\Omega_F^\infty$, let $\mathcal{O}_{F,S}$ be the ring of $S$-integers of $F$, i.e., the set of $a \in F$ such that $|a|_v \leq 1$ for all $v \notin S$.
Let $\overline{F}$ be an algebraic closure of $F$ and write $\Gamma_F$ for the absolute Galois group of $F$.

\subsection{Varieties} For an algebraic variety $X$ over a field $F$ we write $X(F)$ for the set of $F$-rational points of $X$. We let $\Pic(X)$ denote the Picard group of $X$ and let $\Eff^1(X)\subset \Pic(X)_{\R}$ denote the 
cone of pseudo-effective divisors on $X$. We often identify line bundles, Cartier divisors and their classes in $\Pic(X)$. 

\subsection{Algebraic groups} \label{sec:algebraic_groups}

See \cite{Milne} for background and details on the structure theory of reductive groups. Let $\G$ be a connected semi-simple group of adjoint type over $F$, and let $\G'$ be the quasi-split inner form of $\G$. The groups $\G$ and $\G'$ have a unique split $F$-form $\G^{sp}$. Let $E$ be a finite Galois extension of $F$ such that $\G(E) = \G'(E) = \G^{sp}(E)$ considered as subgroups of $\G(\bar{F})$. Let $\T$ be a maximal $F$-torus in $\G$. For each place $v$ of $F$, $\T_{F_v}$ is a maximal torus in $\G_{F_v}$ \cite[Thm.~17.82]{Milne}.   For each $v \in \Omega_F$, let $\S_v \subset \G_{F_v}$ be the unique maximal $F_v$-split subtorus of $\T_{F_v}$ considered as an algebraic group over $F_v$ \cite[Prop.~13.2.4]{Springer}. We denote the corresponding objects for $\G'$ and $\G^{sp}$ by $\T'$, $\S_v'$, and $\T^{sp}$, respectively. We make the choice in such a way that $\T^{sp}$ is split.  We also assume $\T(E) = \T'(E) = \T^{sp}(E)$ as subgroups of $\G(\bar F)$. Let $\S'$ be the maximal split torus in $\G'$ contained in $\T'$.

Let $S_E$ with $\Omega_E^\infty \subset S_E \subset \Omega_E$ be a finite set of places such that for $w \not\in S_E$ we have 
\begin{enumerate}
\item for all $v \in \Omega_F$ with $w \mid v$ the extension $E_w / F_v$ is unramified; 
\item if $v \in \Omega_F$ is not divisible by any $w \in S_E$, then $\G$ is quasi-split over $F_v$ and $\G(F_v) = \G'(F_v)$. 
\end{enumerate}
Let $S_F$ be the collection of places of $F$ that are divisible by some place in $S_E$. 

Fix a Borel subgroup $\B^{sp}$ in $\G^{sp}$ containing $\T^{sp}$. Pick a Borel subgroup $\B'$ in $\G'$ containing $\T'$. For $v \not\in S_F$, by transfer of structure we obtain a Borel subgroup $\B_v$ in $\G_{F_v}$. If $v \in S_F$ we just pick a minimal parabolic $F_v$-subgroup $\B_v$ of $\G_{F_v}$ which contains $\S_v$. For each $v \in \Omega_F$, let $\tilde \Phi(\G_{F_v}, \S_v)$ be the set of roots of $\S_v$ in $\G_{F_v}$ and by $\Phi(\G_{F_v}, \S_v)$ the set of non-multipliable roots with the ordering given by $\B_v$. Let $\mathfrak X^*(\S_v)$ be the set of $F_v$-characters of $\S_v$, and let $\mathfrak X^+$, respectively $\Phi^+$, the set of positive characters, respectively roots, in $\mathfrak X^*(\S_v)$, respectively in $\Phi(\G_{F_v}, \S_v)$, with respect to the ordering.  We also let $\Delta(\G^{sp}, \T^{sp})= \{\overline{\alpha}_1, \dots, \overline{\alpha}_l\}$
be the set of simple roots for the pair $(\G^{sp}, \T^{sp})$, and let $\rho$ be half the sum of positive roots, i.e., 
$$
\rho = \frac{1}{2}\sum_{\overline{\alpha} \in \Phi^+(\G^{sp}, \T^{sp})} \overline{\alpha}. 
$$
We also define constants $\kappa_{\overline{\alpha}}$ for $\overline{\alpha} \in \Delta(\G^{sp}, \T^{sp})$ by 
\begin{equation}\label{equ:tworho}
2 \rho = \sum_{\overline{\alpha} \in \Delta(\G^{sp}, \T^{sp})} \kappa_{\overline{\alpha}} \overline{\alpha}. 
\end{equation}
There is a perfect pairing $\langle , \rangle:\mathfrak{X}^*(\S_v) \times \mathfrak{X}_*(\S_v) \rightarrow \mathbb{Z}$ \cite[Lem.~2.11]{Springer} where $\mathfrak X_*(\S_v)$ denotes the set of cocharacters of $\S_v$.  Since by assumption $\G$ is adjoint, the simple roots form a basis of $\mathfrak X^*(\T)$, and so for each simple root $\theta \in \Delta(\G_{F_v},\S_v)$ we may define a cocharacter $\check{\theta}$ by $\langle \alpha,\check{\theta} \rangle=-\delta_{\alpha \theta},$ where $\delta_{\alpha \theta}$ is 1 if $\alpha=\theta$, and is 0 otherwise.

It also makes sense to talk about the root system of the pair $(\G', \T')$. We have 
$$
\Phi(\G', \T') \subset \mathfrak{X}_E^*(\T')= \mathfrak{X}^*(\T').
$$
This root system has an ordering given by the choice of $\B'$, and the ordering determines a set of simple roots $\Delta(\G', \T')$. The Galois group $\mathrm{Gal}(E/F)$ has a natural action on $\mathfrak{X}^*(\T')$, and this action preserves the root system; see \cite[Ch. 24.b]{Milne} for details.. The set $\Delta(\G', \S')$ is a set of representatives for the collection of $\mathrm{Gal}(E/F)$-orbits of $\Delta(\G', \T')$.  We also have a surjective restriction map 
$$
r: \Phi(\G', \T') \to \Phi(\G', \S'). 
$$
The Galois group acts transitively on the fibers of this map. Similarly, for each place $v$ of $F$,  we can talk about the set $\Delta(\G'_{F_v}, \T'_{F_v})$.  We also have a local restriction map 
$$
r_v: \Phi(\G'_{F_v}, \T'_{F_v}) \to \Phi(\G'_{F_v}, \S'_{F_v}). 
$$
For details regarding the restriction maps $r$ and $r_v$ see \cite[2.8-2.9]{STT}.

\subsection{Cartan decomposition}
Let $F$ be a number field and $v$ be a place of $F$.
We set
$$
\widehat{F}_v := \begin{cases}  \quad
 \{ x \in \R \, |\,  x \geq 1 \},  & v \text{ archimedean} ,\\
  \quad  \{ \varpi_v^{-n} \, | \, n \in \N \}, &  v \text{ non-archimedean},
\end{cases}
$$
and
\begin{equation}
\S_v(F_v)^+ = \{ a \in \S_v(F_v) \, | \, \alpha(a) \in \widehat{F}_v
\text{ for each } \alpha \in \Phi^+(\S_v) \}.
\end{equation}
Here the set $\N$ of natural numbers includes $0$.

\begin{prop}[Cartan decomposition]
\label{prop:Cartan}
For each place $v$, there is a maximal compact subgroup $\K_v$ of $\G(F_v)$ and a finite set $\Omega_v \subset \G(F_v)$ such that $\G(F_v)=\K_v \S_v(F_v)^+\Omega_v \K_v.$ More precisely for each $g \in \G(F_v)$, there exist unique elements $a \in \S_v(F_v)^+$ and $d \in \Omega_v$ such that $g \in \K_v a d \K_v.$ If $F_v$ is archimedean or if $\G$ is unramified over $F_v$ (i.e.~quasi-split, and splits over an unramified extension), then $\G(F_v)=\K_v \S_v(F_v)^+ \K_v.$
\end{prop}

\begin{proof}
See \cite[Section 2.1]{oh}.
\end{proof}

Let $\K_v$ for each place $v$ be a choice of subgroup as in Proposition \ref{prop:Cartan}.

\subsection{Measures} For each nonarchimedean place $v$ of a number field $F$, let $\mathrm dg_v$ denote the Haar measure on $\G(F_v)$ such that $\K_v$ has volume $1$. For each archimedean place $v$, choose any Haar measure $\mathrm dg_v$. Then the collection $\{\mathrm dg_v:v \in \Omega_F\}$ defines a Haar measure, say $\mathrm{d}g$, on $\G(\mathbb{A}_F)$. Since $\G$ is semi-simple,  the volume of $\G(F) \backslash \G(\mathbb{A}_F)$ is finite with respect to the induced measure \cite[Thm.~5.6 (i)]{borel-ihes}. Therefore, by changing $\mathrm dg_v$ for archimedean $v$ if necessary, we may assume that $\vol(\G(F) \backslash \G(\mathbb{A}_F))=1.$
The measure $\mathrm{d}g$ is uniquely determined by this condition, due to uniqueness of Haar measures up to scalars.

\subsection{Automorphic characters} \label{sec:aut_char}
 Let $F$ be a number field. An automorphic character of $\G(\mathbb{A}_F)$ is a continuous homomorphism from $\G(\mathbb{A}_F)$ to the unit circle $\{z \in \mathbb{C}:z \overline{z}=1\}$ whose kernel contains $\G(F)$. Let $\mathscr{X}$ denote the set of all automorphic characters of $\G(\mathbb{A}_F)$. For any compact open subgroup $W_f$ of $\G(\mathbb{A}_f)$, we let $\mathscr{X}^{W_f} \subset \mathscr{X}$ denote the set of all $W_f$-invariant characters in $\mathscr{X}$, i.e., $\mathscr{X}^{W_f}=\{\chi \in \mathscr{X}:\chi(w)=1 \ \text{for all} \ w \in W_f\}$. We set
\[
G_{W_f}=\text{ker}(\mathscr{X}^{W_f}):= \bigcap_{\chi\in \mathscr{X}^{W_f}} \text{ker}(\chi).
\]
By \cite[Lemma 4.7(1)]{GMO} we have $\# \mathscr{X}^{W_f}=[\G(\mathbb{A}_F):G_{W_f}]<\infty.$ 

In particular, each $\chi \in \mathscr{X}^{W_f}$ has finite order. 
Also, by \cite[Lemma 4.7(3)]{GMO}, for any $g \in \G(\mathbb{A}_F)$ we have
\[
\sum_{\chi \in \mathscr{X}^{W_f}}\chi(g)=
\begin{cases}
	\# \mathscr{X}^{W_f}, & \text{if } g \in G_{W_f}, \\
	0, & \text{if } g \not\in G_{W_f}.
\end{cases}
\]

\subsection{Eisenstein series}

We also need the notion of the Eisenstein series $E(s, \phi).$
For its definition, see \cite[Section 4.3.2]{LTBT18}. We will use the notation established in \cite[Section 4.3.2]{LTBT18} in  Sections \ref{sec:HZF}  and \ref{sec:heightII}, especially Corollary \ref{coro:spectralexpansion}.

\section{Campana points on the wonderful compactification}
\label{sec:Campana}

\subsection{The wonderful compactification.} For any semi-simple algebraic group $\G$ with trivial center, De Concini and Procesi \cite{concini-p83} introduced its wonderful compactification $X$. This is a projective variety equipped with an action of $\G \times \G$ and containing $\G$ as an open orbit (where $\G \times \G$ acts on $\G$ by left and right multiplication). 

\subsubsection{Algebraically closed fields}
We review the construction of the wonderful compactification over an algebraically closed field of characteristic $0$.
The reader can consult \cite{concini-p83} for background and details. Let $\G$ be a semi-simple linear algebraic group of adjoint type. We first introduce some notation: choose a Borel subgroup $\B \subset \G$ as well as a maximal torus $\T \subset \B$. Let $\Phi$ denote the root system of $(\G,\T)$, $\Phi^+ \subset \Phi$ the subset of positive roots consisting of roots of $\B$, $\Delta=\{\alpha_1,\dots,\alpha_l\}$ the corresponding set of simple roots, $\rho$ the half-sum of positive roots, and $W$ the Weyl group.

The coroot of any $\alpha \in \Phi$ is denoted by $\alpha^\vee$; this is a cocharacter of $\T$. The coroots form the dual root system $\Phi^\vee$. The pairing between characters and cocharacters is denoted by $\langle , \rangle$; we have $\langle{\alpha,\alpha^\vee \rangle}=2$ for any $\alpha \in \Phi$.

Let $\Lambda$ denote the weight lattice. For any dominant weight $\lambda$, let $V(\lambda)$ denote the irreducible representation of the simply-connected cover of $\G$ having highest weight $\lambda$ \cite[Thm.~22.2]{Milne}. We then have a projective representation
$$\varphi_\lambda:\G \rightarrow \mathrm{PGL}(V(\lambda)).$$
The closure of the image of $\varphi_\lambda$ in the projective space $\mathbb{P}\textrm{End} V(\lambda)$ will be denoted by $X_\lambda$. This is a projective variety on which $\G \times \G$ acts via its action on $\mathbb{P} \textrm{End} V(\lambda)$ by left and right multiplication. If $\omega_\alpha$ for $\alpha \in \Delta$ denote the fundamental weights, then a dominant weight is any weight of the form $\sum_\alpha n_\alpha \omega_\alpha$ where all $n_\alpha>0$. When the dominant weight $\lambda$ is regular, $X_\lambda$ turns out to be smooth and independent of the choice of $\lambda$; this defines the wonderful compactification $X$ of $\G$. 

The boundary $X \backslash \G$ is a union of $l$ irreducible divisors $D_1,\dots,D_l$ with smooth normal crossings. The homomorphism $\varphi_\lambda$ extends to a $\G \times \G$-equivariant morphism $X \rightarrow X_\lambda$ that we shall still denote by $\varphi_\lambda$. For a dominant weight $\lambda$, the pull-back $\mathcal{L}_X(\lambda):=\varphi_\lambda^* \mathcal{O}_{\mathbb{P}\textrm{End}V(\lambda)}(1)$ is a globally generated line bundle on $X$, and all globally generated line bundles on $X$ are of this form. Moreover, $\mathcal{L}_X(\lambda)$ is ample if and only if $\lambda$ is regular. The assignment $\lambda \mapsto \mathcal{L}_X(\lambda)$ extends to an isomorphism $\Lambda \cong \textrm{Pic}(X)$.

We index the boundary divisors so that $\mathcal{O}_X(D_i)=\mathcal{L}_X(\alpha_i)$ for $i=1,\dots,l$. Under this identification, the boundary divisors generate the root lattice. 

\subsubsection{Non-closed fields}
For the description of the wonderful compactification over a non-algebraically closed field we refer to \cite[Section 5.5]{STT}. We take the notation from Section \ref{sec:algebraic_groups}. Let $\Gamma = \mathrm{Gal}(E/F)$.  Then, the pseudo-effective cone $\Eff^1(X)$ is the positive cone spanned by $D_{\Gamma \cdot \beta}$ for $\beta \in \Delta(\G', \T')$, i.e.,
\[\Eff^1(X)=\bigoplus \mathbb{R}_{\geq 0}[D_{\Gamma \cdot \beta}]
\]
where the sum is taken over the $\Gamma$-orbits $\Gamma \cdot \beta$ in the set $\Delta(\G', \T')$ of simple roots and $D_{\Gamma \cdot \beta}:=\sum_{\beta' \in \Gamma \cdot \beta}D_{\beta'}$.  As a consequence the set $\mathcal A$ in Section \ref{sect:mainresult} can be identified with the set of Galois orbits on $\Delta(\G', \T')$. The anticanonical class is given by 
\begin{equation}
\label{equ:anticanonical}
-K_X = \sum_{\alpha \in \mathcal A} (\kappa_{\alpha}+1)D_{\alpha},
\end{equation}
where $\kappa_\alpha$ is as \eqref{equ:tworho} (see \cite[Prop.~5.2]{STT}).
Note that the set $\Delta(\G^{sp}, \T^{sp})$ corresponds to the set $\overline{\mathcal A}$ of geometrically irreducible components of the boundary divisor $D$. As the anticanonical class is Galois invariant, the equation \eqref{equ:tworho} descends to \eqref{equ:anticanonical}.

\subsection{Campana points}
\label{subsec:Campanapoints}

Here we introduce the notion of Campana points given  by Campana \cite{Cam05} and Abramovich \cite{Abra09}. Let $X$ be a smooth projective variety defined over a number field $F$ with a strict normal crossings divisor $\sum_{\alpha} D_\alpha$. A Campana orbifold is a pair $(X, D_\epsilon)$ such that  $D_\epsilon$ is an effective $\mathbb Q$-divisor on $X$ given as $D_\epsilon = \sum_{\alpha \in \mathcal A} \epsilon_\alpha D_\alpha$ with 
\[
\epsilon_\alpha \in \{1-1/m: m \in \mathbb{Z}_{\geq 1}\} \cup \{1\}.
\] 
A Campana orbifold $(X, D_\epsilon)$ is klt if $\epsilon_\alpha <1$ for any $\alpha \in \mathcal A$.

From now on we assume that $(X, D_\epsilon)$ is a Campana orbifold. Choose a finite set of places $S$ of $F$ containing all archimedean places and a regular projective model $(\mathcal{X},\mathcal{D}_\epsilon)$ for $(X,D_\epsilon)$ over Spec $\mathcal{O}_{F,S}$. Here $\mathcal{D}_\epsilon=\sum_{\alpha \in \mathcal{A}}\epsilon_\alpha \mathcal{D}_\alpha,$ where $\mathcal{D}_\alpha$ is the Zariski closure of $D_\alpha$ in $\mathcal{X}$. 

Let $X^\circ=X \backslash \bigcup_{\alpha \in \mathcal{A}, \epsilon_\alpha \neq 0}D_\alpha$. If $P \in X^\circ(F)$ and $v \notin S$, then there is an induced point $\mathcal{P}_v \in \mathcal{X}(\mathcal{O}_v)$. For each $\alpha \in \mathcal{A}$ such that $\mathcal{P}_v \notin \mathcal{D}_\alpha(\mathcal{O}_v)$, the pullback of $\mathcal{D}_\alpha$ via $\mathcal{P}_v$ defines a nonzero ideal in $\mathcal{O}_v$. We let $n_v(\mathcal{D}_\alpha,P)$ denote the colength of this ideal, i.e.~the intersection multiplicity. If $P \in \mathcal D_\alpha(F)$ we set $n_v(\mathcal{D}_\alpha,P) = + \infty$. If $\epsilon_\alpha=1$ we set $m_\alpha=\infty$; otherwise, we set $m_\alpha=1/(1-\epsilon_\alpha)$.

\begin{defn}[{\cite[Definition 3.4]{PSTVA19}}]
\label{defn:Campana_points}
A point $P \in X(F)$ is called a Campana $\mathcal{O}_{F,S}$-point on $(\mathcal{X},\mathcal{D}_\epsilon)$ if the following hold:

\begin{enumerate}

\item we have $P \in (\mathcal X \setminus \cup_{\epsilon_\alpha = 1}\mathcal D_\alpha)(\mathcal O_{F, S})$, and 

\item for $v \notin S$ and all $\alpha$ such that $\epsilon_\alpha<1$ and $n_v(\mathcal D_\alpha,P)>0$ we have $n_v(\mathcal{D}_\alpha,P) \geq m_\alpha$. 
\end{enumerate}
We denote the set of Campana $\mathcal{O}_{F,S}$-points on $(\mathcal{X},\mathcal{D}_\epsilon)$ by $(\mathcal X, \mathcal D_\epsilon)(\mathcal O_{F, S})$.
\end{defn}

Suppose that $X$ is a wonderful compactification of $\G$ with the boundary divisor $D = \sum_{\alpha \in \mathcal A} D_\alpha$.
For any $v \notin S$, the functions $n_v(\mathcal{D}_\alpha, \cdot )$ extend naturally from $\G(F)$ to $\G(F_v)$. Hence we may define the analogous sets $(\mathcal{X},\mathcal{D}_\epsilon)(\mathcal{O}_v)$. Then we define
\[
\G(F_v)_\epsilon=\G(F_v) \cap (\mathcal{X},\mathcal{D}_\epsilon)(\mathcal{O}_v).
\]
For $v \notin S$, we denote by $\delta_{\epsilon,v}$ the indicator function for $(\mathcal{X},\mathcal{D}_\epsilon)(\mathcal{O}_v) \subseteq X(F_v)$. For $v \in S$, we set $\delta_{\epsilon,v}=1$. Let $\delta_\epsilon=\prod_{v \in \Omega_F} \delta_{\epsilon,v}$.

\subsection{Log Manin's conjecture for klt Campana points}
\label{subsec:logManin}

Here we introduce log Manin's conjecture for klt Campana points, proposed in \cite{PSTVA19}.

\subsubsection{Log Manin's conjecture for Fano orbifolds}

Let $(X, D_\epsilon = \sum_{\alpha \in \mathcal A} \epsilon_\alpha D_\alpha)$ be a Campana orbifold defined over a number field $F$.
We say it is a Fano orbifold if $-(K_X + D_\epsilon)$ is ample. 

From now on we assume that $(X, D_\epsilon)$ is a klt Fano orbifold. It follows from \cite[Corollary 1.3.2]{BCHM} that $X$ is a Mori dream space. In particular, the cone of effective divisors $\Eff(X)$ is rationally finitely generated. Let $L$ be a big $\mathbb Q$-divisor on $X$. We define two birational invariants $a((X, D_\epsilon), L), b(F, (X, D_\epsilon), L)$ in the following way:
\begin{equation}
\label{equ:a-invariant}
a((X, D_\epsilon), L) = \min\{t \in \mathbb R \, | \, t[L] + [K_X + D_\epsilon] \in \Eff^1(X)\},
\end{equation}
\begin{equation}
\begin{aligned}
\label{equ:b-invariant}
b(F, (X, D_\epsilon), L) = &\text{ the codimension of the minimal face of } \Eff^1(X)\\
&\text{containing  }a((X, D_\epsilon), L)[L] + [K_X + D_\epsilon].
\end{aligned}
\end{equation}
These are birational invariants (see \cite[Section 3.6]{PSTVA19}.)

Let $S$ be a finite set of places for $F$ including $\Omega_{F}^\infty$. We fix a regular projective model $(\mathcal X, \mathcal D_\epsilon)$ over $\Spec \mathcal O_{F, S}$ of $(X, D_\epsilon)$ as Section~\ref{subsec:Campanapoints}. Let $\mathcal L$ be an adelically metrized big and nef $\mathbb Q$-divisor on $X$. (See Definition~\ref{defn:adelicmetric} for the definition of adelic metrics.) An adelically metrized divisor defines a height function $\mathsf H_{\mathcal L} : X(F) \to \mathbb R_{\geq 0}$. For any subset $Q \subset X(F)$, we define the counting function:
\begin{align*}
N(Q, \mathcal L, T) := \#\{ x \in Q \, | \, \mathsf H_{\mathcal L}(x) \leq T\}.
\end{align*}
Regarding this counting function, we have
\begin{conj}[Log Manin's conjecture, {\cite[Conjecture 1.1]{PSTVA19}}]
\label{conj:logManin}
Assume that the set $(\mathcal X, \mathcal D_\epsilon)(\mathcal O_{F, S})$ is not thin in the sense of \cite[Section 3.4]{PSTVA19}. Then there exist a thin set $Z \subset (\mathcal X, \mathcal D_\epsilon)(\mathcal O_{F, S})$ and a positive constant $c >0$ such that
\begin{equation}
\label{equ:asymptotics}
N((\mathcal X, \mathcal D_\epsilon)(\mathcal O_{F, S})\setminus Z, \mathcal L, T) \sim c T^{a((X, D_\epsilon), L)}(\log T)^{b(F, (X, D_\epsilon), L)-1}
\end{equation}
as $T \to \infty$.
\end{conj}

The interpretation of $c$ as a certain adelic volume is given in \cite[Section 3.3]{PSTVA19} which is inspired by the leading constant in Manin's conjecture for rational points introduced in \cite{Pey95} and \cite{batyrev-t98}. However, this formulation admits a counter example in \cite{Str22} and \cite{Shute21b}. We will propose a refinement of this constant in Section~\ref{sec:leadingconstant}.

\subsubsection{The case of wonderful compactifications}
\label{subsubsec:wonderful}
Let $X$ be a wonderful compactification of a semi-simple algebraic group $\G$ of adjoint type over a number field $F$ with boundary divisor $D = \sum_{\alpha \in \mathcal A} D_\alpha$. Let $(X, D_\epsilon)$ be a klt Campana orbifold.
In the view of (\ref{equ:anticanonical}), our log anticanonical class $-(K_X + D_\epsilon)$ is always big, but not necessary ample. Our theorem holds with this wider setting. Since the cone of effective divisors $\Eff^1(X)$ is a simplicial cone generated by boundary divisors, our definitions $(\ref{equ:a-invariant}), (\ref{equ:b-invariant})$ still apply to arbitrary big $\mathbb Q$-divisors $L$ on $X$. Let $L = \sum_{\alpha \in \mathcal A}\lambda_\alpha D_\alpha$ be a big $\mathbb Q$-divisor with $\lambda_\alpha >0$. Then our invariants $a((X, D_\epsilon), L)$ and $b(F, (X, D_\epsilon), L)$ are given by the following formula:
\begin{align*}
&a((X, D_\epsilon), L) = \max\{ (\kappa_\alpha + 1-\epsilon_\alpha)/\lambda_\alpha \, | \, \alpha \in \mathcal A\} \\
&b(F, (X, D_\epsilon), L) = \# \{ \alpha \in \mathcal A \, | \, a((X, D_\epsilon), L)  = (\kappa_\alpha + 1-\epsilon_\alpha)/\lambda_\alpha \}.
\end{align*}
We will establish Conjecture~\ref{conj:logManin}, especially (\ref{equ:asymptotics}), in this setting.

\section{Height functions}
\label{sec:heights}

In this section we introduce the notion of height functions and the height zeta function for the wonderful compactification $X$.

\subsection{Height functions.} 
We recall various facts about adelically metrized line bundles and height functions, which can be found, for example, in \cite[Sections 2.1, 2.2]{C-T-integral}.
\begin{defn}[$v$-adic metrics]
Let $(X,L)$ be a variety over a number field $F$ together with a line bundle $L$ on $X$. For a place $v \in \Omega_F$, a $v$-adic metric on $L$ is a map which assigns to every point $x_v \in X(F_v)$ a function $||.||_v:L(x_v) \rightarrow \mathbb{R}_{\geq 0}$ on the fiber of $L$ at $x_v$ such that
\begin{enumerate}
\item For all $l \in L(x_v)$, we have $||l||_v=0$ if and only if $l=0$.
\item For all $\lambda \in F_v$ and $l \in L_{x_v}(F_v)$, we have $||\lambda l||_v=|\lambda|_v||l||_v$.
\item For any open subset $U \subset X$ and any local section $s \in \Gamma(U,L)$, the function on $U(F_v)$ given by $x_v \mapsto ||s(x_v)||_v$ is continuous in the $v$-adic topology.
\end{enumerate}
We say a $v$-adic metric on $L$ is smooth if for any open subset $U \subset X$ and any nowhere vanishing section $s \in \Gamma(U, L)$, the function on $U(F_v)$ given by $x_v \mapsto ||s(x_v)||_v$ is $C^\infty$ if $v$ is archimedean and is locally constant if $v$ is non-archimedean.

For any non-archimedean place $v$ of $F$, there is a natural way to associate a smooth $v$-adic metric on $L$ with any integral model $(\mathcal{X},\mathcal{L})$ of $(X,L)$ over $\Spec \, \mathcal{O}_v$. See \cite[Section 2.1.5]{C-T-integral} for more details.

\end{defn}

\begin{defn}[(Smooth) adelic metrics]
\label{defn:adelicmetric}
Let $(X,L)$ be a projective variety over a number field $F$ together with a line bundle $L$ on $X$. A (smooth) adelic metric on $L$ is a collection $||.||=\{||.||_v\}_{v \in \Omega_F}$ of (smooth) $v$-adic metrics for each place $v \in \Omega_F$, such that all but finitely many of the $||.||_v$ are defined by a single projective model of $(X,L)$ over $\mathcal{O}_{F, S}$ for some $S$ a finite set of $\Omega_F$ including all archimedean places. 

\end{defn}

\label{sect:heights}
\subsection{Height functions on the wonderful compactification}

Let us come back to the original situation, i.e., $X$ is the wonderful compactification of a connected semi-simple group $\G$ of adjoint type defined over a number field $F$.
The boundary $X \backslash \G$ decomposes into a union of irreducible divisors:
\[
D=X \backslash \G = \bigcup_{\alpha \in \mathcal{A}}D_\alpha.
\]
For each $\alpha \in \mathcal{A}$, we fix a smooth adelic metric on the line bundle $\mathcal{O}(D_\alpha)$.
This induces an adelic metrization for any line bundle $L$. Let $\mathsf{f}_\alpha$ be a global section of $\mathcal{O}_X(D_\alpha)$ corresponding to a boundary divisor $D_\alpha$. 
For each place $v$, we define the local height pairing by
\[
\mathsf H_v: \textrm{Pic}(X)_\mathbb{C} \times \G(F_v) \rightarrow \mathbb{C}^\times, \ \left(\sum_{\alpha \in \mathcal{A}}s_\alpha D_\alpha, x \right) \mapsto \prod_{\alpha \in \mathcal{A}}||\mathsf{f}_\alpha(x)||_v^{-s_\alpha}.
\]
This pairing varies linearly on the factor Pic$(X)_\mathbb{C}$ and continuously on the factor $\G(F_v)$. We define the global height pairing $\H$ as the product of the local height pairings
\[
\mathsf H=\prod_{v \in \Omega_F} \mathsf H_v: \textrm{Pic}(X)_\mathbb{C} \times \G(\mathbb{A}_F) \rightarrow \mathbb{C}^\times.
\]
Again, this pairing varies linearly on the first factor and continuously on the second factor.

From now on we will use notations introduced in Section~\ref{sec:algebraic_groups}.
For all but finitely many places $v$ of $F$, the group $\G_{F_v}$ is quasi-split \cite[Thm.~6.7]{pr}. Let $v$ be a place of $F$ for which $\G_{F_v}$ is quasi-split.  Then each tuple $\bold{s} = (s_\alpha)_{\alpha \in \mathcal A}$ gives rise to a tuple $\bold{s} = (s^v_\theta)_{\theta \in \Delta(\G_{F_v}, \S_{v})}$. Indeed, when $\theta$ is above $\alpha$, we define as $s^v_\theta = s_\alpha$. When there is no danger of confusion we simply write $\mathbf{s} = (s_\theta)_{\theta \in \Delta(\G_{F_v}, \S_{v})}$. We use these complex variables to perform our local computations without comment.

\begin{lem}
\label{lem:height_k-invariant}
For each nonarchimedean place $v$ of $F$, there is a compact open subgroup $\K_v \subset \G(F_v)$ such that $\mathsf H_v(\bold{s},.)$ is bi-$\K_v$-invariant for all $\bold{s}$. Moreover, for all non-archimedean $v$, one can choose $\K_v$ to satisfy the Cartan decomposition (as in Proposition~\ref{prop:Cartan}), and for almost all $v$ one can choose  $\K_v = \G(\mathcal O_v)$.
\end{lem}

\begin{proof}
This property of the height is shown in \cite[Proposition 6.3]{STT}. 
\end{proof}

Let $v$ be a place such that $\G(F_v)=\K_v \S_v^+ \K_v$, as in the Cartan decomposition (Proposition \ref{prop:Cartan}), and let $g_v \in \G(F_v)$. Write $g_v=k_v t_v k'_v$ with $t_v \in \S_v(F_v)^+$ and $k_v,k'_v \in \K_v$. For each $\theta \in \Delta(\G_{F_v},\S_v)$, let 
\[
l_v(\theta)=\#\{\beta \in \Delta(\G'_{F_v} ,\T'_{F_v}) \ \text{with} \ \beta|_{\S_v}=\theta\}.
\]

\begin{lem}
\label{lem:heightatv}
For almost all places $v$, 
\[
\mathsf H_v(\bold{s},g_v)=\prod_{\theta \in \Delta(\G_{F_v},\S_v)}|\theta(t_v)|_v^{l_v(\theta) s_\theta}.
\]
\end{lem}

\begin{proof}
See (6.4) in \cite{STT}.
\end{proof}

We can also calculate the indicator function $\delta_{\epsilon,v}$ in terms of the Cartan decomposition.

\begin{prop} 
\label{prop:deltaatv}
For all $v \not\in S$, there exists a compact open subgroup $\K_v \subset \G(F_v)$ such that $\delta_{\epsilon,v}$ is invariant under the left and right action of $\K_v$. 
Moreover, for almost all places $v \not\in S$ we have $\delta_{\epsilon,v}(g_v) = 1$ if and only if 
\begin{itemize}
\item for all $\alpha \in \mathcal{A}$ corresponding to $\theta \in \Delta(\G_{F_v},\S_v)$ with $\epsilon_\alpha < 1$ we have $|\theta(t_v)|_v =1$ or $|\theta(t_v)|_v \geq q_v^{m_\alpha}$, and;
\item for all $\alpha \in \mathcal{A}$ corresponding to $\theta \in \Delta(\G_{F_v},\S_v)$ with $\epsilon_\alpha = 1$ we have $|\theta(t_v)|_v =1$.
\end{itemize}
\end{prop}

\begin{proof}
For all places $v \not\in S$, we consider the metrization of $\mathcal O(D_\alpha)$ which is induced by the integral model $\mathcal X$. (This may not be equal to the metrization we fix.) For such a $v$, we have
\[
\mathsf H_{v}'(D_\alpha, g_v) = q_v^{n_v(\mathcal D_\alpha, g_v)},
\]
where $\mathsf H_{v}'(D_\alpha, g_v)$ is the local height function induced by $\mathcal X$.
In particular the first statement follows from this and Lemma~\ref{lem:height_k-invariant}.

On the other hand, for the wonderful compactification $X$, the height function is given by
\[
\mathsf H_{v}'(D_\alpha, g_v) = |\theta(t_v)|^{l_v(\theta)}_v,
\]
for almost all non-archmedean places $v$ by Lemma~\ref{lem:heightatv}. 
When $D_\alpha$ is not geometrically integral, then there is no $k_v$-point on $\mathcal D_\alpha$ because the reduction of $(\mathcal X, \mathcal D)$ at $v$ is still smooth with a strict normal crossings divisor.
In particular we have $\mathsf H_{v}'(D_\alpha, g_v)  = 1$ for any $g_v$.
When $D_\alpha$ is geometrically irreducible, then we have
\[
\mathsf H_{v}'(D_\alpha, g_v) = |\theta(t_v)|_v.
\]
Thus our assertion follows.
\end{proof}

\subsubsection{Fixing $\K_v$.}

For every non-archimedean place $v$, we now fix the choice of such a $\K_v$ satisfying Proposition~\ref{prop:Cartan}, Lemma~\ref{lem:height_k-invariant}, and Proposition~\ref{prop:deltaatv} such that  $\K_v = \G(\mathcal O_v)$ for all but finitely many $v$. For archimedean $v$, let $\K_v$ be such that $\G(F_v)=\K_v \S_v(F_v)^+ \K_v$ as in the Cartan decomposition. Let $\K=\prod_v \K_v$ and let $\K_0=\prod_{v < \infty}\K_v.$

\subsection{Height zeta function} \label{sec:HZF}

By a Tauberian theorem, to establish the asymptotic formula it suffices to establish analytic properties of the height zeta function
\[
\mathsf Z(sL)=\sum_{\gamma \in \G(F)_\epsilon} \mathsf H(L, \gamma)^{-s}.
\]
In fact we will consider the function $\mathsf Z_\epsilon:\textrm{Pic} (X)_\mathbb{C} \times \G(\mathbb{A}_F) \rightarrow \mathbb{C}$ defined by
\[
\mathsf Z_\epsilon(\bold{s},g)=\sum_{\gamma \in \G(F)_\epsilon}\mathsf H(\bold{s}, \gamma g)^{-1}=\sum_{\gamma \in \G(F)} \mathsf H(\bold{s},\gamma g)^{-1}\delta_\epsilon(\gamma g).
\]
Observe that if we set $g=e$, the identity element in $\G(F)$, we get $\mathsf Z_\epsilon(\bold{s},e)=\sum_{\gamma \in \G(F)_\epsilon} \mathsf H(\bold{s},\gamma)^{-1}$.
The next result obtains a spectral decomposition of the height zeta function.

\begin{prop} \label{prop:spectral_decomposition}
 Given $g \in \G(\mathbb{A}_F)$, the series defining $\mathsf Z_\epsilon(.,g)$ converges absolutely to a holomorphic function for $\bold{s} \in \mathcal{T}_{\gg 0}.$ For all $\bold{s}$ in the region of convergence we have $\mathsf Z_\epsilon(\bold{s},.) \in C^\infty(\G(F) \backslash \G(\mathbb{A}_F))$, and for all integers $n \geq 1$ and all $\partial \in \mathfrak U(\mathfrak{g})$ we have $\partial^n \mathsf Z_\epsilon(\bold{s},.) \in \mathsf L^2(\G(F)\backslash \G(\mathbb A_F))$ where $\mathfrak U(\mathfrak{g})$ is the universal enveloping algebra of $\mathfrak g$. Moreover
\begin{align*}
\mathsf Z(\bold{s},e) & =\sum_{\chi \in \mathscr{X}} \int_{\G(\mathbb{A}_F)} \delta_\epsilon(g) \mathsf H(\bold{s},g)^{-1}\chi(g) \ \mathrm dg \\
& + \sum_{\chi \in \mathfrak{X}}\sum_P n(A_P)^{-1}\int_{\Pi(M_P)}\left(\sum_{\phi \in B_P(\pi)_\chi}E(e,\phi)\int_{G(\mathbb{A}_F)}\delta_\epsilon(g) \mathsf H(\bold{s},g)^{-1}\overline{E(g,\phi)} \ \mathrm dg\right) \ \mathrm d \pi_p.
\end{align*}
Here this is the spectral expansion described in \cite[Section 4.4]{LTBT18}.

\end{prop}

Note that the notation in the above formula is explained in \cite[Section 4.4]{LTBT18}, and we do not repeat those in this paper.

\begin{proof}
This is a consequence of \cite[Lemma 3.1]{STT}. Explicitly, by the proof of \cite[Theorem 7.1]{STT}, $\mathsf H(\bold{s}, g)$ satisfies the conditions of Lemma 3.1 of \cite{STT}, i.e., with $n, N$ as in \cite[Lemma 3.1]{STT} the integrals
$$
\int_{\G(\mathbb{A}_F)} |\mathsf H(\bold{s}, g)| \| g \|^N \, dg, \quad
\int_{\G(\mathbb{A}_F)} |\bDelta^n \mathsf H(\bold{s}, g)| \| g \|^N \, dg,  
$$
with $\bDelta, \| \, . \, \|$ as in \cite[Section 3.2]{STT}, converge.  These integrals still converge if we replace $\mathsf H(\bold{s}, g)$ by $\mathsf H(\bold{s}, g)\delta_\epsilon(g)$. We now apply \cite[Lemma 3.1]{STT} to get the result. 
\end{proof}
In the following sections, we will compute the rightmost pole and the order of the pole. In particular, we will show that, as in the prior cases of rational and integral points, the contributions to the right most pole come from one-dimensional representations. 

\section{Height integrals I: one-dimensional representations}
\label{sec:heightI}

In this section, we will establish analytic properties of the infinite product
\[
\int_{\G(\mathbb{A}_F)} \delta_\epsilon(g) \mathsf H(\bold{s},g)^{-1}\chi(g) \, \mathrm dg =\prod_{v \in \Omega_F}\int_{\G(F_v)} \mathsf H_v(\bold{s},g_v)^{-1} \chi_v(g_v) \delta_{\epsilon,v}(g_v) \, \mathrm dg_v
\]
where $\chi \in \mathscr{X}$ is an automorphic character of $\G$.
In this section it will be convenient for us to not necessarily assume that $D_\epsilon$ is klt, i.e. we allow $\epsilon_\alpha = 1$. The klt condition is only relevant to the analysis of the infinite dimensional representations treated in Section \ref{sec:heightII}. We begin with the following basic result.

\begin{lem} \label{lem:Fourier_vanish}
Suppose $\Re \mathbf s \gg 0$. If $\chi_v$ is not $\K_v$-invariant, then 
$$
\int_{\G(F_v)} \mathsf H_v(\bold{s},g_v)^{-1} \chi_v(g_v) \delta_{\epsilon,v}(g_v) \, \mathrm dg_v=0. 
$$
\end{lem}
\begin{proof}
The assumption $\Re \mathbf s \gg 0$ guarantees the absolute convergence of the integral. Pick  $\kappa_v \in \K_v$ such that $\chi_v(\kappa_v) \ne 1$. Then making a change of variable $g_v \mapsto g_v\kappa_v$ and using the $\K_v$-invariance of $\mathsf H_v$ and $\delta_{\epsilon, v}$, gives 
$$
\int_{\G(F_v)} \mathsf H_v(\bold{s},g_v)^{-1} \chi_v(g_v) \delta_{\epsilon,v}(g_v) \, \mathrm dg_v= \chi_v(\kappa_v) 
\int_{\G(F_v)} \mathsf H_v(\bold{s},g_v)^{-1} \chi_v(g_v) \delta_{\epsilon,v}(g_v) \, \mathrm dg_v. 
$$
The result is now obvious. 
\end{proof}

In particular we need to only consider automorphic characters $\chi=\otimes'_v \chi_v \in \mathscr{X}^{\K_0}$ where $\mathscr{X}^{\K_0}$ is the group of automorphic characters which are invariant under the action of $\K_0$ and $\otimes'_v$ means the restricted tensor product. 

\subsection{Hecke characters}\label{subsect:heckecharacters}
We associate to each automorphic character $\chi$ of $\G$ a collection of Hecke $L$-functions $L(s, \chi_\alpha)$ for $\alpha \in \mathcal{A}$,
 as follows. 
 This construction is from Section 2.8 of \cite{STT} which we now briefly recall. 

Let $\G'$ denote the (unique up to isomorphism) quasi-split form of $\G$. The Galois group of $E$ over $F$ acts on the set $\Delta(\G',\T') $ 
of simple roots via  the natural action on the set of characters of $\T'$. Let $\alpha \in \mathcal A$. Then $\alpha$ is an orbit of the Galois action on $\Delta(\G', \T')$. Fix a simple root $\beta \in \alpha$. 
and let $F_\beta$ be the field of definition of $\beta$, i.e., the fixed field of the stabilizer of $\beta$ under the Galois action. Then the cocharacter $\check{\beta}:\mathbb G_m \rightarrow \T$ is defined over $F_\beta$. It induces a continuous homomorphism $\check{\beta}_\mathbb{A}:\mathbb{A}_{F_\beta}^\times \rightarrow \T(\mathbb{A}_{F_\beta})$. For each $\beta$ we have the following norm map:
for each place $v$ of $F$ and place $u$ of $F_\beta$ that lies over $v$, we define 
\[
N_{u/v}:\T((F_\beta)_u) \rightarrow \T(F_v), \ \ \ t \mapsto \prod_{\sigma \in \Gal((F_\beta)_u/F_v)}\sigma(t).
\]
Next we define
\[
N_v:\prod_{u|v}\T((F_\beta)_u) \rightarrow \T(F_v), \ \ \ (t_u)_u \mapsto \prod_{u|v} N_{u/v}(t_u).
\]
Finally, we define 
\[
N_{F_\beta}:\T(\mathbb{A}_{F_\beta}) \rightarrow \T(\mathbb{A}_F), \ \ \ (t_u)_u \mapsto (N_v(t_u))_v.
\]
For each character $\chi$ of $\T(\mathbb{A}_F)/\T(F)$, the map $\chi_\beta$ defined by
\[
\chi_\beta = \chi \circ N_{F_\beta} \circ \check{\beta}_\mathbb{A}:\mathbb{A}^\times_{F_\beta} \rightarrow S^1
\]
is a Hecke character of $F_\beta$. For $u \in \Omega_{F_\beta}$, we define $\chi_{\beta,u}$ by $\chi_\beta=\prod_{u \in \Omega_{F_\beta}}\chi_{\beta,u}.$ The Hecke $L$-function $L(s,\chi_\beta)$ depends only on the Galois orbit $\alpha$. For this reason, we denote the $L$-function $L(s,\chi_\beta)$ by $L(s,\chi_\alpha)$. Write 
$$
L(s, \chi_\alpha) = \prod_u L_u (s, \chi_{\beta, u}). 
$$
\subsection{Good primes}
We call a place $v$ {\em good} if it satisfies the following conditions: 
\begin{itemize} 
\item $v \not\in S$ .
\item $v$ is such that our model $\mathcal X$ has good reduction at $v$ and our metrizations are induced by this model. 
\item Lemma~\ref{lem:heightatv} and Proposition~\ref{prop:deltaatv} are valid for $v$. 
\item $v \not\in S_F$ with $S_F$ as in \cite[Lemma 6.4]{STT}.  
\item The character $\chi_v$ is unramified. 
\end{itemize} 
Let $S'(\chi)$ be a finite set of places such that any $v \not\in S'(\chi)$ is good.

\

Now suppose $v$ is a good place for $\chi$, and let $u$ be a place of $F_\beta$ such that $u \mid v$.  Suppose $(F_\beta)_u/F_v$ is unramified and that $\chi_{\beta,u}$ is unramified. Let $\varpi_u$ be a prime element of $(F_\beta)_u$. Then
\[
L_u(s,\chi_{\beta,u})=(1-\chi_{\beta,u}(\varpi_u)q_u^{-s})^{-1}.
\]
Next, we view $\beta$ in $\Delta(\G'_{F_v},\T'_{F_v})$. The restriction of $\beta$ to $\S'_v$ is an element $\theta_v \in \Delta(\G'_{F_v},\S'_v)$. For $\theta \in \Delta(\G'_{F_v},\S'_v)$, let 
\[
l_v(\theta)=\#\{\beta \in \Delta(\G'_{F_v},\T'_{F_v}) \ \textrm{with} \ \beta|_{\S'_v}=\theta\}.
\]

\begin{lem}
\label{lem:local} 
\[
L_u(s,\chi_{\beta,u})= (1-\chi_v(\check \theta_v(\varpi_v))q_v^{-l_v(\theta_v)s})^{-1}.
\]
\end{lem}

\begin{proof}
This is Proposition 2.9 in \cite{STT}.
\end{proof}

The following result generalizes the computations in 
\cite[Sections 6 and 7]{STT} to the context of Campana points.
For a character $\rho$  of $\mathbb{A}_F^\times/F^\times$  we define the partial $L$-function
\[
L^{S'(\chi)}(s,\rho)=\prod_{v \notin S'(\chi)}L_v(s,\rho_v),
\]
where the local factor is 
\[
L_v(s,\rho_v)=(1-\rho_v(\varpi_v)q_v^{-s})^{-1}
\]
for unramified $v$. Let $\rho$ be a Hecke character of $\mathbb{A}_F^\times/F^\times$. 
We also denote $\mathcal A_{\mathrm{klt}} = \{ \alpha \in \mathcal A \, | \, \epsilon_\alpha < 1\}$.
Recall the definition of $(\kappa_\alpha)_{\alpha \in \mathcal{A}}$ from \eqref{equ:anticanonical}.

\begin{thm}\label{thm:one-dim}
We have
$$
\!
\int_{\G(\A_{S'(\chi)})}  \delta_\epsilon(g)  \mathsf H_{S'(\chi)}({\bf s}, g)^{-1}\chi_{S'(\chi)}(g) \, \mathrm dg = f_{S'(\chi),\chi}(\bold{s}) \prod_{\alpha \in \mathcal{A}_{\mathrm{klt}}}L^{S'(\chi)}(m_\alpha(s_\alpha-\kappa_\alpha),\chi_\alpha^{m_\alpha})$$
where $f_{S'(\chi),\chi}$ is a function given an absolutely convergent Euler product on the domain
$$
\left\{(s_\alpha)_{\alpha \in \mathcal A}  \mid {\Re s_\alpha  > \kappa_\alpha + 1/(m_\alpha +1) \text{ for all} \ \alpha \in \mathcal{A}}\right\}.
$$
In particular $f_{S'(\chi),\chi}$ is uniformly bounded on compact sets and holomorphic in this domain.
Moreover for the trivial character, the function $f_{S'(1),1}$ is non-zero in its domain for $s_\alpha$ real.
\end{thm}

\begin{proof}
Write $\chi_{S'(\chi)}= \otimes'_{v\not\in S'(\chi)}\chi_v$. The collection of elements $\{
\check\theta(\varpi_v)\}_{\theta \in \Delta(\G_{F_v},\S_v)}$ 
forms a basis for the semi-group $\S_v(F_v)^+$.
For ${\bf a} = (a_\alpha)_{\alpha \in \mathcal{A}}$ where $a_\alpha$ is an integer we set
\begin{equation}
t_v({\bf a}) = \prod_{\theta \in \Delta(\G_{F_v},\S_v)}\check\theta(\varpi_v)^{a_\theta} \in  \S_v(F_v)^+.
\end{equation}
If $\alpha \in \mathcal{A}$ corresponds to $\theta \in \Delta(\G_{F_v},\S_v)$, then we set $s^v_\theta=s_\alpha$. For each $v$ for which $\G_{F_v}$ is quasi-split, we define
\[
\langle {\bf s}, {\bf a}\rangle_v = \sum_{\theta \in \Delta(\G_{F_v},\S_v)}  s^v_\theta a^v_\theta l_v(\theta).
\]
Using bi-$\K_v$-invariance and Lemma~\ref{lem:heightatv}, we express each local integral as an infinite sum:
\begin{equation}
\label{eqn:rewrite}
\int_{\G(F_v)}\delta_{\epsilon,v}
(g_v) \mathsf H_v({\bf s}, g_v)^{-1} \chi_v(g_v) \, \mathrm dg_v =
\sum_{{\bf a} \in \N^r} \delta_{\epsilon,v}(t_v({\bf a})) 
q_v ^{-\langle{\bf s,  \bf a\rangle}_v}  \chi_v(t_v({\bf
a}))\text{vol}(\sK_v t_v({\bf a})\sK_v).
\end{equation}
Let $\mathrm{d}_{\B}(t_v({\bf a})):=q_v^{\langle{ 2\rho, \bf a \rangle}_v}.$ By \cite[Lemma 6.11]{STT}, there is a constant $C$, independent of $v$, 
such that for all
$t_v({\bf a}) \in \S_v (F_v)^+$, one has
\begin{equation}
\label{eqn:volume}
\vol(\sK_v t_v({\bf a})  \sK_v) \leq \mathrm{d}_\B(t_v({\bf a}))\left(1 + \frac{C}{q_v}\right).
\end{equation}
We define
\begin{equation*}
b_v ({\bf s})= \sum_{{\bf a}\in \N^\mathcal{A}} \delta_{\epsilon,v}(t_v({\bf a})) 
q_v ^{-\langle{\bf s, \bf a
\rangle}_v}  \chi_v(t_v({\bf
a}))(\text{vol}(\sK_vt_v({\bf a})\sK_v)- \mathrm{d}_\B(t_v({\bf a})))
\end{equation*}
and see that Equation~\eqref{eqn:rewrite} equals
\begin{equation*}
\sum_{{\bf a} \in \N^\mathcal{A}} \delta_{\epsilon,v}(t_v({\bf a}))
q_v ^{-\langle{\bold{s}-2\rho,\bf a \rangle}_v} \chi_v(t_v({\bf a})) + b_v({\bf s}).
\end{equation*}
Observe that using Proposition~\ref{prop:deltaatv} we have
\begin{equation}\begin{split}\label{delta}
& \sum_{\bf a\in \N^\mathcal{A}} \!\delta_{\epsilon,v}(t_v({\bf a})) q_v^{-\langle{\bold{s}-2\rho,\bold{a} \rangle}_v} 
\chi_v(t_v({\bf a})) \\
& = \prod_{\theta \in \mathcal{A}_{\mathrm{klt}}, l_v(\theta) = 1} \left( 1+ 
\sum_{a_\theta = m_\theta}^\infty \chi_v(\check\theta(\varpi_v))^{a_\theta}q_v ^{-(s_\theta - \kappa_{\theta}) a_\theta }  \right) \\
& =\prod_{\theta \in \mathcal{A}_{\mathrm{klt}}, l_v(\theta) = 1} \left( 1 +  { \chi_v(\check\theta(\varpi_v))^{m_\theta}q_v^{-(s_\theta -
\kappa_{\theta}) m_\theta} \over 1-
\chi_v(\check\theta(\varpi_v))q_v^{-(s_\theta-
\kappa_{\theta})}}\right)\\
& = \prod_{\theta\in \mathcal{A}_{\mathrm{klt}}, l_v(\theta) = 1} {1 \over 1-
\chi_v(\check\theta(\varpi_v))^{m_\theta}q_v^{-(s_\theta-
\kappa_{\theta})m_\theta}} \\
& \times  \prod_{\theta \in \mathcal{A}_{\mathrm{klt}}, l_v(\theta) = 1} \left(1-
\chi_v(\check\theta(\varpi_v))^{m_\theta}q_v^{-(s_\theta -
\kappa_{\theta})m_\theta }\right) \cdot \left( 1 +  { \chi_v(\check\theta(\varpi_v))^{m_\theta}q_v^{-(s_\theta -
\kappa_{\theta}) m_\theta} \over 1-
\chi_v(\check\theta(\varpi_v))q_v^{-(s_\theta -
\kappa_{\theta})}}\right) \\
& = \prod_{\theta \in \mathcal{A}_{\mathrm{klt}}, l_v(\theta) = 1} {1 \over 1-
\chi_v(\check\theta(\varpi_v))^{m_\theta}q_v^{-(s_\theta -
\kappa_{\theta})m_\theta }}\\ 
& \times \prod_{\theta \in \mathcal{A}_{\mathrm{klt}}, l_v(\theta) = 1} \left( 1 + { \chi_v(\check\theta(\varpi_v))^{m_\theta+1}q_v^{-(s_\theta -
\kappa_{\theta})(m_\theta+1) } -  \chi_v(\check \theta(\varpi_v))^{2m_\theta}q_v^{-(s_\theta -
\kappa_{\theta})2  m_\theta }\over 1-
\chi_v(\check\theta(\varpi_v))q_v^{-(s_\theta -
\kappa_{\theta})}} \right) 
\end{split}\end{equation}
By \cite[Section 2]{STT} there is an automorphic character $\chi'$ of $\G'(\mathbb{A}_F)$ such that for almost all $v$, 
\[
L_u(s,(\chi'_\alpha)_u)=(1-\chi_v(\check \theta_v(\varpi_v))q_v^{-l_v(\theta_v)s})^{-1}.
\]
Therefore, the corresponding Euler product for $\alpha \in \mathcal{A}$ over $v \notin S'(\chi)$ is a product of partial $L$-functions, as in the statement of the theorem. The last product over $\alpha \in \mathcal{A}$ and $v \not\in S'(\chi)$ converges whenever for all $\alpha \in \mathcal{A}$ we have $\Re s_\alpha > \kappa_\alpha + 1/(m_\alpha+1)-\epsilon$ for some $\epsilon>0$.  

\

Let $\mathsf{\sigma}= (\Re(s_\al))_\al$.
In the definition $b_v$, we may assume ${\bf a} \ne
\underline{0}$. Since for each $v \notin S$,
$$
\Biggl\{ {\bf a}\,\, | \,\, {\bf a} \ne \underline{0} \Biggr\} =
\bigcup_{\theta \in \Delta(\G_{F_v},\S_v)} \Biggl\{ {\bf
a}; a_\theta \ne 0 \Biggr\},
$$
we have
\begin{align*}
\sum_{v \notin S'} \Bigl| b_v({\bf s}) \Bigr| & \leq \sum_{v \notin S}
\sum_\theta \sum_{a_\theta \ne 0}\delta_{\epsilon,v}(t_v({\bf a})) 
q_v ^{{-\langle {\bf \sigma},  {\bf a}
 \rangle}_v} \Biggl|
(\text{vol }(\sK_vt_v({\bf a})\sK_v)-  \mathrm{d}_{\B}(t_v({\bf a})))\Biggr|  \\
& \ll\sum_{ v \notin S} q_v^{-1} \left(\sum_{\theta \in \mathcal{A}}
\sum_{a_\theta \ne 0} q_v ^{{-\langle  {\bf
\sigma},  {\bf a}  \rangle}_v}
\mathrm{d}_{\B}(t_v({\bf a})) \right)  \\
& = \sum_{v \not\in S}q_v^{-1} \sum_{\theta \in \mathcal{A}} \left(\sum_{a_\theta = m_\theta}^\infty q_v^{-(\sigma_\theta -
\kappa_{\theta})a_\theta l(\theta)} \right) \prod_{{\beta \ne \theta \atop 
\beta \in \mathcal{A}}} \left( 1 + \sum_{a_\beta = m_\beta}^\infty q_v^{-(\sigma_\beta - \kappa_{\beta})a_\beta l(\beta)} \right) \\
& =\sum_{v \not\in S}q_v^{-1} \sum_{\theta \in \mathcal{A}} \frac{q_v^{-(\sigma_\theta -
\kappa_{\theta})m_\theta}}{1-q_v^{-(\sigma_\theta -
\kappa_{\theta})l(\theta)}}\prod_{{\beta \ne \theta \atop 
\beta \in \mathcal{A}}} \left( 1 + \frac{q_v^{-(\sigma_\beta -
\kappa_{\beta})l(\beta) m_\beta}}{1 -q_v^{-(\sigma_\beta -
\kappa_{\beta})l(\beta)} }\right) < \infty 
\end{align*}
provided $\sigma_\alpha > \kappa_\alpha$ for each $\alpha \in \mathcal A$.

Note that if $\Re s_\alpha >\kappa_\alpha + 1/(m_\alpha+1) + \epsilon$ for all $\alpha \in \mathcal{A}$, then all estimates are uniform and the corresponding function $f_{S'(\chi), \chi}$ is given by an absolutely convergent Euler product.  As a result $f_{S'(\chi), \chi}$ satisfies the stated analytic properties.

The final part regarding the trivial character follows from the above calculations and non-vanishing of the local integrals.
\end{proof}

\subsection{Bad primes}
Here we assume that $v \in S'(\chi) \setminus S$.

\begin{thm}\label{thm:one-dim-bad-primes}
The function 
\[
\int_{\G(F_v)}  \delta_{\epsilon, v}(g)  \mathsf H_{v}({\bf s}, g)^{-1}\chi_{v}(g) \, \mathrm dg
\]
is holomorphic and uniformly bounded in 
$
\left\{(s_\alpha)_{\alpha \in \mathcal A}  \mid {\Re s_\alpha  > \kappa_\alpha \text{ for all} \ \alpha \in \mathcal{A}}\right\}.
$
\end{thm}
\begin{proof}
Let $\mathcal X^\circ = \mathcal X \setminus (\cup_{\alpha \not\in \mathcal A_{\mathrm{klt}}} \mathcal D_\alpha)$. 
The function in the statement is given by
\[
\int_{\mathcal X^\circ (\mathcal O_v)}  \delta_{\epsilon, v}(g)  \mathsf H_{v}({\bf s}, g)^{-1}\chi_{v}(g) \, \mathrm dg
\]
which is bounded by
\[
\int_{\mathcal X^\circ (\mathcal O_v)}    \mathsf H_{v}(\Re{\bf s}, g)^{-1} \, \mathrm dg.
\]
Then note that the function
\[
\prod_{\alpha \not\in \mathcal A_{\mathrm{klt}}}\mathsf H_{v}(\Re s_\alpha D_\alpha, g)^{-1} 
\]
is a nowhere vanishing locally constant function on $\mathcal X^\circ(\mathcal O_v)$. Thus using arguments in \cite[Lemma 4.1]{C-T-integral}, one can prove that the above integral is absolutely converges when 
\[
{\Re s_\alpha  > \kappa_\alpha \text{ for all }  \alpha \in \mathcal{A}}. \qedhere
\]

\end{proof}

\subsection{Places in $S$}

In this case, $\delta_{\epsilon,v} = 1$ by definition. Therefore the local height integral for Campana points coincides with the usual height integral, so that we do not need to do anything new.

\begin{prop}
\label{ramified} 
Let 
$
\mathcal{J}_v(\bfs) := \int_{\G(F_v)} \mathsf H_v(\bfs, g_v)^{-1} \chi_v(g_v) \, \mathrm dg_v.
$
\begin{enumerate} \item For $v\notin
S_{\infty}$ the function 
\[
\prod_{\alpha \in \mathcal{A}}L_v(s_\alpha-\kappa_\alpha,\chi_\alpha)^{-1} \mathcal{J}_v(\bfs)
\]
is holomorphic in the domain
\[
 \mathcal T_{-1-\delta} = \left\{(s_\alpha)_{\alpha \in \mathcal A}  \mid {\Re s_\alpha  > \kappa_\alpha-\delta, \text{for all} \ \alpha \in \mathcal{A}}\right\}.
\]
for some $\delta > 0$. 
\item For $v\in S_{\infty}$ and $\partial$ in
the universal enveloping algebra, the integral
\begin{equation*}
\prod_{\alpha \in \mathcal{A}}L_v(s_\alpha-\kappa_\alpha,\chi_\alpha)^{-1}\mathcal{J}_{v,\partial}(\bfs):= \prod_{\alpha \in \mathcal{A}}L_v(s_\alpha-\kappa_\alpha,\chi_\alpha)^{-1}\int_{\G(F_v)} \partial( \mathsf H_v(\bfs,
g_v)^{-1} ) \chi_v(g_v) \,\mathrm  d g_v
\end{equation*}
is holomorphic for $\bfs \in \cT_{-1-\delta}$,
for some $\delta > 0$.
\end{enumerate}
Moreover, for all $\epsilon > 0$ and $\partial$, there exist 
constants $C_v(\eps)$ and  $C_v(\partial, \eps)$ such that 
$$
\left|\prod_{\alpha \in \mathcal{A}}L_v(s_\alpha-\kappa_\alpha,\chi_\alpha)^{-1}\mathcal{J}_v(\bfs)\right| \leq
C_v(\eps) \quad \text{ and } \quad \left|\prod_{\alpha \in \mathcal{A}}L_v(s_\alpha-\kappa_\alpha,\chi_\alpha)^{-1}\mathcal{J}_{v,\partial}(\bfs)\right| \leq C_v(\partial, \eps),
$$
for all  $\bfs \in \cT_{-1 - \eps}$ where 
\[
\cT_{-1 - \eps} = \left\{(s_\alpha)_{\alpha \in \mathcal A}  \mid {\Re s_\alpha  > \kappa_\alpha - \epsilon, \text{for all} \ \alpha \in \mathcal{A}}\right\}.
\]
\end{prop} 

\begin{proof}
Let $\mathsf H_{0,v}(\bold s, g_v) : \G(F_v) \to \mathbb C^\times$ be the special local height function which is invariant by a maximal compact subgroup $\prod_{v \in \Omega_F} \K_v$ defined in \cite[Section 1]{TBT13}. Then the function 
\[
c(\bold s, g_v) := \frac{\mathsf H_v(\bold s, g_v)}{\mathsf H_{0, v}(\bold s, g_v)}
\]
extends to a locally constant and non-vanishing function on $X(F_v)$. In particular for a fixed $\bold s$ the function $c(\bold s, x_v)$ takes only finitely many values.
Now our integral takes the form
\[
\mathcal{J}_v(\bfs) := \int_{\G(F_v)} c(\bold s, g_v)\mathsf H_{0,v}(\bfs, g_v)^{-1} \chi_v(g_v) \, \mathrm dg_v.
\]
Thus our assertion follows from the proofs of \cite[Theorem 3.1 and Theorem 4.4]{TBT13}.
\end{proof}

This completes our analysis of one-dimensional representations in the spectral expansion of the height zeta function.

\section{Height integrals II: Infinite-dimensional representations} 
\label{sec:heightII}

In this section we will prove that, as in the case of rational points, the remaining terms in the spectral decomposition do not contribute to the rightmost pole.  Here we adapt the proof of \cite[Theorem 4.10]{LTBT18}. In order to do this we just need to prove an analogue of \cite[Corollary 7.4]{STT} in our context. Throughout this section, we assume that the quasi-split inner form of $\G$ has rank $\geq 2$. 
The results hold even if $(X, D_\epsilon)$ is not klt, though our bounds are only useful in the klt case (see Theorem \ref{important-theorem}).

\

Let $\pi = \otimes_v' \pi_v$ be an infinite-dimensional automorphic representation of $\G(\A_F)$.
Let $S'(\pi)\supset S$ be a finite set of places such that $v\not\in S'(\pi)$ is a place such that our metrizations are induced by this model $\mathcal{X}$, $\G_{F_v}$ is quasi-split, $\K_v = \G(\mathcal O_v)$, $\pi_v$ is unramified, and we have the standard Cartan decomposition $G(F_v) = \K_v\S_v(F_v)^+\K_v$. For $v \not\in S'(\pi)$, we let $\varphi_{\pi_v}$ the associated normalized spherical function. (See \cite{Cassel80} for the definition of normalized spherical functions.)

\

The following gives bounds on spherical functions. Note in particular that the right-hand side does not depend on the automorphic representation $\pi$ or an automorphic form $\phi$. 
\begin{prop}
\label{prop:spherical}
Let $n$ be the absolute rank of $\G$.  Suppose $\G$ is not the inner form of a quasi-split group of rank $1$. 
Then for all $\epsilon>0$ there is a constant $C_\epsilon>0$ such that for all $v \not\in S'(\pi)$ and $g_v = k_v t_v d k'_v$ with $k_v, k'_v \in \K_v, t_v \in \S_v(F_v)^+, d \in \Omega_v$, with notations as in Proposition \ref{prop:Cartan}, we have

$$|\varphi_{\pi,v}(g_v)| \leq C_\epsilon \prod_{\theta \in \Delta (\G_{F_v},\S_v)} |\theta(t_v)|_v^{-\frac{1}{3n}+\epsilon}.$$
\end{prop}

\begin{proof}
Let $r$ be the split rank of $\G_{F_v}$.  
Recall that $\{\check{\theta}(\varpi_v)\}_{\theta \in \Delta(\G_{F_v},\S_v)}$ forms a basis for $\S_v(F_v)^+$. Each $\{\theta\}$ is a strongly orthogonal system of $\Phi$ (see \cite[3.3]{oh} for the definition).
Since the rank of $\G_{F_v}$ is at least two, Theorem 1.1 and Theorem 5.9(3) of \cite{oh} 
imply that for every $\epsilon>0$, there is a constant $C_\epsilon$ such that, with notation as in Proposition \ref{prop:Cartan}, for all $g_v = k_1 t_v d k_2 \in \G(F_v)$ we have
\[
|\varphi_{\pi_v}(g_v)| \leq C_\epsilon |\theta(t_v)|_v^{-1/2+\epsilon r}.
\]
Multiplying these inequalities over all $\theta \in \Delta(\G_{F_v},\S_v)$ gives
\[
|\varphi_{\pi_v}(g_v)|^r \leq C_\epsilon^r \prod_{\theta \in \Delta(\G_{F_v},\S_v)}|\theta(t_v)|_v^{-1/2+\epsilon r}.
\]
Taking the $r$th root and noting $r \leq n$ gives the result. 
\end{proof}

Set
\begin{equation}
\mathcal{J}_v({\bf s}, \pi_v) =
\int_{\G(F_v)}\delta_{\epsilon, v}(g_v)  \mathsf H_v({\bf s}, g_v)^{-1}\varphi_{\pi_v}(g_v)\,\mathrm  dg_v.
\end{equation}

\begin{thm}
\label{non-triv} 
Suppose $\G$ is not the inner form of a quasi-split group of rank $1$. 
The infinite product
\begin{equation}\label{non-triv-eq}
\mathcal{J}_{S'(\pi), \cD}({\bf s}, \pi) : = \prod_{v \notin S'(\pi)}
\mathcal{J}_v({\bf s}, \pi_v)
\end{equation}
is holomorphic on $\Re s_\alpha > \kappa_\alpha + 1/m_\alpha - 1/3n$ for all $\alpha \in \mathcal{A}$. 
Here $n$ is as in Proposition \ref{prop:spherical}. For all $\eps>0$ and all compact subsets $$K \subset \{{\bf s} \in \mathrm{Pic}(X)_{\mathbb C} \, | \, \Re  s_\alpha > \kappa_\alpha + 1/m_\alpha - 1/3n+\epsilon \ \text{for all} \ \alpha \in \mathcal{A}\},$$
there is a constant $C(\eps,K)$, independent of $\pi$, such that
\begin{equation*}
\vert \mathcal{J}_{S'(\pi), \cD}({\bf s}, \pi) \vert \leq C(\eps,K), \quad \text{for all }{\bf s} \in K.
\end{equation*}
\end{thm}
\begin{proof}

Recall that for ${\bf a} = (a_\alpha)_{\alpha \in \mathcal{A}}$ we set
\begin{equation}
t_v({\bf a}) = \prod_{\theta \in \Delta(\G_{F_v},\S_v)}\check\theta(\varpi_v)^{a_\theta} \in  \S_v(F_v)^+.
\end{equation}
Using bi-$\K$-invariance and Lemma~\ref{lem:heightatv}, as in the proof of Theorem~\ref{thm:one-dim}, 
we obtain

\begin{align*}
\mathcal{J}_v({\bf s}, \pi_v) &= \sum_{{\bf a}\in \N^r} 
\delta_{\epsilon,v}(t_v({\bf a})) \mathsf H_v(\bold{s},t_v(\bold{a}))^{-1}
\varphi_{\pi_v}(t_v({\bf a})) \text{vol}(\K_v t_v({\bf a}) \K_v) \\
 &= \sum_{{\bf a}\in \N^r} 
\delta_{\epsilon,v}(t_v({\bf a})) q_v^{-\langle{\bold{s},\bold{a}\rangle}_v}
\varphi_{\pi_v}(t_v({\bf a})) \text{vol}(\K_v t_v({\bf a}) \K_v).
\end{align*}
Using the estimates from \eqref{eqn:volume} and Proposition~\ref{prop:spherical} 
we conclude that to establish the convergence of the infinite product it suffices to bound
$$
 \prod_{v \not\in S'(\pi)} \left( \sum_{\bf a\in \N^{\mathcal{A}}} \!\delta_{\epsilon,v}(t_v({\bf a})) q_v^{-\langle{\bold{s}-2\rho,\bold{a} \rangle}_v} \prod_{\theta \in \Delta(\G_{F_v},\S_v)}|\theta(t_v({\bf a}))|_v^{-\frac{1}{3n}+\epsilon}\right).
$$
Using Proposition~\ref{prop:deltaatv} we see that this is equal to 
$$
 \prod_{v \not \in S'(\pi)}\prod_{\alpha \in \mathcal{A}, l_v(\alpha) = 1}  \left( 1+ 
\sum_{a_\alpha = m_\alpha}^\infty q_v ^{-(s_\alpha - \kappa_{\alpha}+ {1 \over 3n}-\epsilon) a_\alpha}  \right).
$$
The sums are calculated as geometric series. The above expression becomes 
$$
\prod_{v \not \in S'(\pi)} \prod_{\alpha \in \mathcal{A}, l_v(\alpha) = 1}  \left( 1 +  { q_v^{-(s_\alpha -
\kappa_{\alpha}+ {1 \over 3n}-\epsilon)m_\alpha} \over 1-
q_v^{-(s_\alpha -
\kappa_{\alpha} + {1 \over 3n}-\epsilon)}}\right). 
$$
This infinite product converges whenever $\Re s_\alpha > \kappa_\alpha + 1/m_\alpha - 1/3n+\epsilon$ for $\alpha \in \mathcal{A}.$  Since $\epsilon >0$ is arbitrary, the result follows. 
\end{proof}

Let $v \in S'(\pi) \setminus S$ and we consider integrals of the form
\begin{equation*}
\mathcal{J}_v(\bfs) := \int_{\G(F_v)}\delta_{\epsilon, v}(g_v)
\mathsf H_v(\bfs, g_v)^{-1} \, \mathrm dg_v.
\end{equation*}

\begin{prop}
\label{ramified1}
For $v \in S'(\pi) \setminus S$, $\mathcal{J}_v(\bfs, \varphi_{\pi_v})$ is holomorphic for ${\bf s}$ such that $\Re s_\alpha > \kappa_\alpha$ for all $\alpha \in \mathcal{A}$. 
\end{prop}
\begin{proof}
Our assertion follows from \cite[Lemma 4.1]{C-T-integral}.
\end{proof}

Now we consider integrals of the form
\begin{equation*}
\mathcal{J}_v(\bfs) := \int_{\G(F_v)}
\mathsf H_v(\bfs, g_v)^{-1}\, \mathrm dg_v, \quad \text{for }v\in S.
\end{equation*}


\begin{thm}
\label{ramified2}  \hfill
\begin{enumerate} \item For all $v \in
S$ the integral $ \mathcal{J}_v(\bfs)$ is 
holomorphic for ${\bf s}$ such that $\Re s_\alpha > \kappa_\alpha$ for all $\alpha \in \mathcal{A}$.
\item For archimedean $v\in S$ and $\partial$ in
the universal enveloping algebra, the integral
\begin{equation*}
\mathcal{J}_{v,\partial}(\bfs):= \int_{\G(F_v)}
\partial( H_v(\bfs, g_v)^{-1} ) \, \mathrm d g_v
\end{equation*}
is holomorphic for ${\bf s}$ such that $\Re s_\alpha > \kappa_\alpha $ for all $\alpha \in \mathcal{A}$.
\end{enumerate}

\end{thm}
\begin{proof} 
Statement (1) follows from \cite[Lemma 4.1]{C-T-integral}.
Similarly (2) follows from a combination of \cite[Proposition 2.2]{chambert-t02} and \cite[Lemma 4.1 and Proposition 4.2]{C-T-integral}
\end{proof}

\begin{coro}
\label{coro:spectralexpansion}
In the non-archimedean situation, if $v \in S'(\pi)$, for each $\epsilon > 0$ there is a constant $C_v(\eps)$, such that $|\mathcal{J}_v(\bfs)| \leq C_v(\eps)$ for all $\mathbf{s}$ such that $\Re s_\alpha > \kappa_\alpha + \epsilon$ for all $\alpha \in \mathcal{A}$.
In the archimedean situation,  for all $\eps>0$ and all $\partial$ as above, there is a constant $C_v(\partial, \eps)$ such that
$|\mathcal{J}_{v,\partial}(\bfs)| \leq C_v(\partial,
\eps)$ for all such that $\Re s_\alpha > \kappa_\alpha + \epsilon$ for all $\alpha \in \mathcal{A}$.
\end{coro}
\begin{proof}
This follows from Proposition~\ref{ramified1} and Theorem~\ref{ramified2}.
\end{proof}

Let 
$$
S^\flat(\bold{s})= \sum_{\chi \in \mathfrak{X}}\sum_P n(A_P)^{-1}\int_{\Pi(M_P)}\left(\sum_{\phi \in B_P(\pi)_\chi}E(e,\phi)\int_{G(\mathbb{A}_F)}\delta_\epsilon(g) \mathsf H(\bold{s},g)^{-1}\overline{E(g,\phi)} \, \mathrm dg\right) \, \mathrm d \pi_p.
$$
be the contribution of non-one-dimensional representations to the spectral expansion described in \cite[Section 4.4]{LTBT18}.
All together we conclude the following theorem:
\begin{thm}
\label{important-theorem}
Suppose $\G$ is not the inner form of a quasi-split group of rank $1$. 
The function $S^\flat$ admits an analytic continuation to a function
which is holomorphic on ${\bf s}$ such that $\Re s_\alpha > \max\{\kappa_\alpha + 1/m_\alpha - 1/3n, \kappa_\alpha\}$ for $\alpha \in \mathcal{A}$, where $n$ is the absolute rank of $\G$. 

In particular, if $D_\epsilon$ is klt, then $S^\flat$ is holomorphic on $\Re(\mathbf s) \in (\Eff(X) -(K_X+D_\epsilon) - \delta D )^\circ$ for some $\delta > 0$.
\end{thm}
\begin{proof}
The proof is identical to the proof of Theorem 4.10 of \cite{LTBT18}. The key ingredient is the analogue of \cite[Theorem 4.9]{LTBT18} in our setting and this theorem, similar to \cite[Corollary 7.4]{STT}, follows from \cite[Lemma 4.7]{LTBT18},  Theorem \ref{non-triv},  Proposition \ref{ramified1}, and Corollary \ref{coro:spectralexpansion}.  Here for the convenience of the reader we provide a very brief outline of the proof. For each archimedean place $v$, let $\K_v$ be a maximal compact subgroup as in Proposition \ref{prop:Cartan}. Let $\G_\infty = \prod_{v \text{ arch.}} \G(F_v)$ and $\K_\infty = \prod_{v \text{ arch.}} \K_v$. For each finite dimensional smooth representation $W$ of $\K_\infty$, let $\mathrm{tr}_W$ be its trace function extended to $\G_\infty$ by defining it to be zero outside $\K_\infty$. We let $\xi_W = (\mathrm{dim}\, W) \cdot \mathrm{tr}_W$.   For each non-archimedean place $v$, pick $\K_v$ as in Proposition \ref{prop:deltaatv}, and let $\xi_v = \mathrm{vol}\, (\K_v)^{-1} ch_{\K_v}$ where here $ch_{\K_v}$ is the characteristic function of the set $\K_v$. Let $\xi_f = \otimes_{v \text{ non-arch}} \xi_v$. By a theorem of Harish-Chandra \cite[Prop. 4.4.3.2.]{Warner}, $\sum_{W \in \hat \K_\infty} (\xi_f \otimes \xi_W )* (\delta_\epsilon(g) \H(\mathbf{s}, g)^{-1})$ converges in the topology of $C^\infty(\G(\A))$ to $\delta_\epsilon(g) \H(g)$.  Now we have 
\begin{align*}
\int_{G(\mathbb{A}_F)}\delta_\epsilon(g) \mathsf H(\bold{s},g)^{-1}\overline{E(g,\phi)} \, \mathrm dg & = \sum_{W \in \hat\K_\infty}\int_{G(\mathbb{A}_F)}(\xi_f \otimes \xi_W )* (\delta_\epsilon(g) \H(\mathbf{s}, g)^{-1})\overline{E(g,\phi)} \, \mathrm dg \\ 
& = \sum_{W \in \hat\K_\infty}\int_{G(\mathbb{A}_F)} (\delta_\epsilon(g) \H(\mathbf{s}, g)^{-1})\overline{(\xi_f \otimes \xi_W )*E(g,\phi)} \, \mathrm dg, 
\end{align*}
after a change of variables.  The key point is that $(\xi_f \otimes \xi_W )*E(.,\phi)$, modulo some harmless functions, is the product $\prod_{v \not\in S} \varphi_{\pi_v}$ where here $\pi = \otimes_v' \pi_v$ is the automorphic representation for $E(., \phi)$ and $\varphi_{\pi_v}$ is the normalized spherical function associated to the representation $\pi_v$ (for details see Lemma 4.7 of \cite{LTBT18}).  Now that we have spherical functions we can use the bounds obtained in Theorem \ref{non-triv} for $v \not\in S$ and 
Proposition \ref{ramified1}, and Corollary \ref{coro:spectralexpansion} for $v \in S$. Once we use these bounds, the convergence of the resulting sum is a result of M\"{u}ller \cite[Corollary 0.3]{Muller}; for details see the proof of Theorem 4.10 of \cite{LTBT18}. 

In the klt case we have $1/m_\alpha > 0$, thus  $\max\{\kappa_\alpha + 1/m_\alpha - 1/3n, \kappa_\alpha\} < \kappa_\alpha + 1/m_\alpha$.
\end{proof}

\begin{coro}
\label{coro:holo}
Suppose $\G$ is not the inner form of a quasi-split group of rank $1$. 
The height zeta function $\mathsf Z_\epsilon(\mathbf s) = \mathsf Z_\epsilon(\mathbf s, e)$ is holomorphic for $\Re(\mathbf s) \in (\Eff(X) -(K_X+D_\epsilon))^\circ.$ 
\end{coro}
\begin{proof}
	Follows immediately from the spectral decomposition 
	(Proposition \ref{prop:spectral_decomposition}), Theorems \ref{thm:one-dim}, \ref{thm:one-dim-bad-primes}, \ref{important-theorem}, and Proposition \ref{ramified}.
\end{proof}

\begin{rem}\label{rem:rank1}
Theorem \ref{theo:maintheorem}, should be true for the rank $1$ case. An issue that arises for rank one groups is that  Proposition \ref{prop:spherical} is not true for quasi-split rank one groups, as in such cases the trivial representation is not isolated in the unitary dual of the local group. As a consequence, one cannot follow the argument presented above to prove a version of Theorem \ref{non-triv} for rank one groups. For cuspidal automorphic representations, well-known approximations to the Ramanujan conjecture, as explained in the proof of Theorem 4.5 of \cite{STT}, give a version of Proposition \ref{prop:spherical} and consequently a statement similar to Theorem \ref{non-triv}. Automorphic representations associated to Eisenstein series on rank one groups do not satisfy the Ramanujan conjecture, and for that reason no analogue of Proposition \ref{prop:spherical} is valid for them. 

For rank one groups, the method of \cite{PGL2} can be adapted to give the proof of the main theorem. Let us briefly describe what is involved.  Lemma 3.1 of \cite{PGL2} needs to be adjusted as in our Theorem \ref{thm:one-dim}.  This takes care of one-dimensional representations. Similarly, the key ingredient in the proof of Lemma 3.2 of \cite{PGL2}, Equation (3.5), needs to be adapted analogously to the computations in the proof of Theorem \ref{non-triv}. Once the analogue of Lemma 3.2 has been proved, the proof of Lemma 7.5 of \cite{PGL2} goes through without change, and that takes care of the contribution of the cuspidal spectrum. Finally, the proof of Proposition 5.1 of \cite{PGL2} needs to be modified slightly, as Equation (5.3) in the proof of that proposition is based on Equation (3.5) of that paper which,  as indicated above, needs to be modified. Once the analogue of Proposition 5.1 is established, the treatment of the contribution of Eisenstein series is identical to what is already in \cite{PGL2}. 
\end{rem}

\section{The leading pole}
\label{sect:poles}
We now prove Theorem \ref{theo:maintheorem}.
We recall our setting. Let $F$ be a number field and $\G$ be a semi-simple group of adjoint type defined over $F$. We assume that $\G$ is not the inner form of a quasi-split group of rank $1$ (this is to apply the results from Section \ref{sec:heightII}). Let $X$ be the wonderful compactification of $\G$ defined over $F$.
We denote the boundary as
\[
D = X \setminus \G = \bigcup_{\alpha \in \mathcal A} D_\alpha.
\]
Let $(X, D_\epsilon = \sum_\alpha \epsilon_\alpha D_\alpha)$ be a klt Campana orbifold. 
Let $S$ be a finite set of places of $F$ including $\Omega_F^\infty$, and we fix a regular model of $X$ away from $S$ as $\mathcal X \to \Spec \, \mathcal O_{F, S}$. We are interested in the set of Campana points
\[
\G(F)_\epsilon = (\mathcal X, \mathcal D_\epsilon)(\mathcal O_{F, S})\cap \G(F).
\]
We fix a smooth adelic metric
on $\mathcal O(D_\alpha)$ for each $\alpha \in \mathcal A$ and we consider the height paring
\[
\mathsf H : \Pic(X)_{\mathbb C} \times \G(\mathbb A_F)   \to \mathbb C^\times.
\]
As in Section \ref{sec:HZF}, we consider the height zeta function by
\[
\mathsf Z_\epsilon(\bfs, g) := \sum_{\gamma \in \G(F)} \delta_\epsilon(\gamma g)\mathsf H(\bfs, \gamma g)^{-1}.
\]
Let $L = \sum_{\alpha} \lambda_\alpha D_\alpha$ be a big divisor on $X$. 
Our goal is to identify the rightmost pole of $\mathsf Z_\epsilon(sL, \mathrm{id})$, where $s \in \C$ and we view $sL \in \Pic(X)_{\mathbb C}$ in the obvious way.

The expected location of the right most pole is
\[
a=\text{max}\left\{ \frac{\kappa_\alpha+1-\epsilon_\alpha}{\lambda_\alpha} :  \alpha \in \mathcal A \right\}.
\]
Let $\mathcal{A}(L)$ be the set of $\alpha \in \mathcal{A}$ such that $\frac{\kappa_\alpha+1-\epsilon_\alpha}{\lambda_\alpha}=a$.
The expected order of the pole is $b = \# \mathcal{A}(L)$.
Note that we have $a = a((X, D_\epsilon), L)$ and $b = b(F, (X, D_\epsilon), L)$ as explained in Section~\ref{subsubsec:wonderful}. We define 
\begin{equation} \label{eqn:def_X(G)}
\mathscr X_{D_{\epsilon}, L}(\G) = \{ \chi \in \mathscr X (\G)^{\K_0} : \chi_\beta^{m_\alpha} = 1
\text{ for all } \alpha \in \mathcal{A}(L), \text{ and } \beta \in \alpha\}.
\end{equation}
where $\mathscr X (\G)^{\K_0}$ denotes the collection of $\K_0$-invariant automorphic characters. These will be the characters which can contribute towards the leading singularity.
The following is a more precise version of Theorem \ref{theo:maintheorem}.

\begin{thm} \label{thm:leading_constant}
Let $a = a(X, D_{\epsilon}, L), b = b(F, X, D_{\epsilon}, L)$. Then we have
	\begin{equation*}
	\mathsf N (\G(F)_\epsilon, \mathcal L,B) = c\cdot B^a \log(B)^{b-1}(1+o(1)),   \quad B\ra \infty,
	\end{equation*}
	where
	\[
	c =\frac{ |\mathscr{X}_{D_{\epsilon}, L} (\G)|}{a(b-1)!}
	 \lim_{s \to a} (s-a)^b \int_{\G(\mathbb A_F)^{ \mathscr{X}_{D_{\epsilon}, L} (\G) }}\delta_\epsilon(g)\mathsf H(sL, g)^{-1} \mathrm dg
	\]
	is positive and $\G(\mathbb A_F)^{\mathscr X_{D_{\epsilon}, L} (\G)} = \bigcap_{\chi \in \mathscr X_{D_{\epsilon}, L} (\G)} \mathrm{Ker}(\chi).$
	
\end{thm}

\begin{proof}
We prove this by applying a Tauberian theorem \cite[Thm.~III]{Delange} to our height zeta function
$
\mathsf Z_{\epsilon}(sL, \mathrm{id}).
$
This is holomorphic when $\Re(s) > a((X, D_{\epsilon}), L)$ by Corollary~\ref{coro:holo}.
Moreover by Theorem~\ref{important-theorem}, only one-dimensional automorphic characters can contribute to the rightmost pole in the spectral decomposition (Proposition \ref{prop:spectral_decomposition}). It follows from Theorem~\ref{thm:one-dim} that a possible pole has order at most $b(F, (X, D_{\epsilon}), L)$.
We analyze which $\chi$ can contribute to the leading pole. Firstly we note from Lemma \ref{lem:Fourier_vanish} that we need only consider $\K_0$-invariant characters. Moreover by Theorem~\ref{thm:one-dim}, the integral
\[
\int_{\G(\mathbb A_F)} \delta_\epsilon(g)\mathsf H(sL, g)^{-1}\chi(g) \mathrm dg
\]
can have the highest order pole of $b(F, (X, D_{\epsilon}), L) $ only if $\chi \in \mathscr X_{D_{\epsilon}, L} (\G)$. As the group of such characters is finite, character orthogonality applied to the spectral sum shows that the main term of $\mathsf Z_{\epsilon}(sL, \id) $ is given by
\[
| \mathscr X_{D_{\epsilon}, L} (\G)| \int_{\G(\mathbb A_F)^{\mathscr X_{D_{\epsilon}, L} (\G)}}\delta_{\epsilon}(g)\mathsf H(sL, g)^{-1} \mathrm dg.
\]
A Tauberian theorem now gives the stated formula. It only remains to prove that $c > 0$.

Part of the difficulty with showing this is that the integral cannot be written as an Euler product in general due to the presence of the characters $\chi$, even away from a finite set of places (this will cause all kinds of problems with formulating Conjecture \ref{conj:leading_constant}.) Therefore we instead restrict the integral to a smaller adelic space which trivialises the characters and yields an Euler product. This  method is inspired by the work of Streeter in \cite[Prop.~6.11]{Str22}. Let $S' \supset S$ be a finite set of places such that $S'$ includes all $v \in \Omega_F$ such that $\G_{F_v}$ is not quasi-split. We replace our Campana indicator function $\delta_\epsilon$ by a different function, given by $\delta'_\epsilon = \prod_v \delta'_{\epsilon,v}$ where for all $v \notin S'$ we define
$$\delta'_{\epsilon,v}(P_v) = 1 \quad \iff \quad 
\begin{array}{ll} 
&\text{ for all $\alpha \in \mathcal{A}(L)$  we have $m_\alpha \mid n_v(\mathcal{D}_\alpha,P)$} \\
& \text{ and for all $\alpha \notin \mathcal{A}(L)$  we have $n_v(\mathcal D_\alpha,P_v) = 0,$}
\end{array}  $$
and for $v \in S'_{\mathrm{fin}}$ we define
$$\delta'_{\epsilon,v}(P_v) = 1 \quad \iff \quad P_v \in \K_v.$$
(Compare with Definition \ref{defn:Campana_points} where we instead require $n_v(\mathcal{D}_\alpha,P) \geq m_\alpha$.) We obviously have
$$\lim_{s \to a} (s-a)^b \int_{\G(\mathbb A_F)^{ \mathscr{X}_{D_{\epsilon}, L} (\G) }}\delta_\epsilon(g)\mathsf H(sL, g)^{-1} \mathrm dg
\geq \lim_{s \to a} (s-a)^b \int_{\G(\mathbb A_F)^{ \mathscr{X}_{D_{\epsilon}, L} (\G) }}\delta'_\epsilon(g)\mathsf H(sL, g)^{-1} \mathrm dg.$$
We claim that 
\begin{equation} \label{eqn:claim_delta'}
	\delta'_{\epsilon,v}(g_v)\chi_v(g_v) = \delta'_{\epsilon,v}(g_v)
\end{equation}
for all $\chi \in \mathscr{X}_{D_{\epsilon}, L}(\G)$ and all $g = (g_v) \in \G(\A_F)$ with $g_v \in \K_v$ for all $v \in S'_{\mathrm{fin}}$. For $v \in S'$ this is clear as each such $\chi_v$ is $\K_v$-invariant by definition. So take $v \notin S'$. Here the group $\G_{F_v}$ is quasi-split. We proceed as in the proof of Theorem \ref{thm:one-dim}. By the Cartan decomposition (Proposition \ref{prop:Cartan}), given $g_v \in \G_{F_v}$ there are elements $k_{1,v}, k_{2,v} \in \K_v$ and $t_v \in  \S_v(F_v)^+$ such that $g_v = k_{1,v} t_v k_{2_v}$.  Then there is a Galois-invariant\footnote{Here, Galois-invariant means if $\alpha_1, \alpha_2$ are in the same orbit of the local Galois group, then $a_{\alpha_1, v} = a_{\alpha_2, v}$.}  ${\bf a} = (a_{\alpha, v})_{\alpha \in \mathcal{A}}$, with $a_{\alpha,v} \geq 0$ integers, such that 
\begin{equation}\label{eq:tv}
t_v = \prod_{\theta_v \in \Delta(\G_{F_v},\S_{F_v})}\check\theta_v(\varpi_v)^{a_{\theta_v}} \in  \S_v(F_v)^+,
\end{equation}
where the exponent $a_{\theta_v}$ is equal to $a_{\alpha, v}$ for any $\alpha \in \mathcal A$ that restricts to $\theta_v$.\footnote{Note that there is a slight abuse of language here: The element $\alpha$ is a Galois orbit of roots. However, If some $\beta \in \alpha$ restricts to $\theta_v$, then every $\beta \in \alpha$ restricts to $\theta_v$.} 

If $\delta'_{\epsilon,v}(g_v) = 0$ then the claim \eqref{eqn:claim_delta'} is clear. So assume that $\delta'_{\epsilon,v}(g_v) = 1$. It suffices to show that $\chi_v(g_v) =1$. We translate the condition $\delta'_{\epsilon,v}(g_v) = 1$ in terms of the expression \eqref{eq:tv}. Let $\theta_v \in \Delta(\G_{F_v}, \S_{v})$. If $\theta_v$ is the restriction of an element $\alpha \in \mathcal A(L)$ then we must have $m_\alpha \mid a_{\theta_v}$, otherwise we have $a_{\theta_v}=0$.  We denote the set of elements $\theta_v \in \Delta(\G_{F_v}, \S_{v})$ that appear as the restriction of an element of $\mathcal A(L)$ as $\mathcal A (L)_{F_v}$.  For $\theta_v \in \mathcal A(L)_{F_v}$ we write $a_{\theta_v} = m_\alpha b_{\theta_v}$ for some integer $b_{\theta_v}$.  Writing $g_v =k_{1,v} t_v k_{2_v}$ with $k_{1,v}, k_{2,v} \in \K_v$ and $t_v \in  \S_v(F_v)^+$ via the Cartan decomposition, using \eqref{eq:tv} we have
$$
\chi_v(g_v) = \chi_v(t_v)= \chi_v \left(\prod_{\theta_v \in \Delta(\G_{F_v},\S_{F_v})}\check\theta_v(\varpi_v)^{a_{\theta_v}}\right)
= \prod_{\theta_v \in \mathcal A(L)_{F_v}} \chi_v (\check \theta_v(\varpi_v))^{m_\alpha b_{\theta_v}}.
$$
For a $\theta_v \in \mathcal A(L)_{F_v}$, let $\alpha \in \mathcal A(L)$ restrict to $\theta_v$. Let $\beta \in \alpha$ and let $w \mid v$ be a place of the field $F_\beta$ as in Section \ref{subsect:heckecharacters}. 
Then by the proof of Proposition 2.9\footnote{That proposition is Lemma \ref{lem:local} of the present paper.}  of \cite{STT}, for each $x \in F_v^\times$ we have
$$
\chi_v (\check \theta_v(x)) = \chi_{\beta, w}(x). 
$$
However for $\chi = \otimes_v' \chi_v \in \mathscr X_{D_{\epsilon}, L}(\G)$ we have
$$
\chi_v (\check \theta_v(\varpi_v))^{m_\alpha b_{\theta_v}} = \chi_{\beta, w}(\varpi_v)^{m_\alpha b_{\theta_v}} = 1
$$
since the latter character is assumed in \eqref{eqn:def_X(G)} to be $m_\alpha$-torsion.
This establishes \eqref{eqn:claim_delta'}.

Using \eqref{eqn:claim_delta'}, we complete the proof as follows. We can lower bound our integral by
\[
\lim_{s \to a} (s-a)^b \int_{\prod_{v \in S'_\infty} \G(F_v)^0 \times \prod_{v \in S'_{\mathrm{fin}}}\K_v \times \G(\mathbb A_{S'}) }\delta'_\epsilon(g)\mathsf H(sL, g)^{-1} \mathrm dg,
\]
where $\G(F_v)^0$ is the identity component of $\G(F_v)$.
However this integral is equal to the Fourier transform of $\delta'_\epsilon$ at the trivial character, up to finitely many places which contribute positively. A minor modification of the proof of Theorem \ref{thm:one-dim} shows that this Fourier transform equals $\prod_{\alpha \in \mathcal{A}(L)} \zeta_{F_\alpha}(m_\alpha(s_\alpha-\kappa_\alpha))$ up to a holomorphic and uniformly bounded function which is non-zero at $s= a$. As this product of Dedekind zeta functions has a pole of order $b$ at $s = a$, we conclude that the limit is non-zero, as required.
\end{proof}

\section{Conjecture for the leading constant}
\label{sec:leadingconstant}
In \cite[Section~3.3]{PSTVA19} a conjecture was put forward for the leading constant for counting Campana points of bounded height. This conjecture was shown to be incorrect independently by Streeter \cite{Str22} and Shute \cite{Shute21b}. The problematic part of the conjecture is the factor $\#\mathrm{H}^1(\Gamma_F, \Pic \bar{U})$ proposed as a replacement for the factor $\beta(X) = \#\mathrm{H}^1(\Gamma_F, \Pic \bar{X})$ in Manin's conjecture. We suggest a fix which is compatible with Theorem \ref{thm:leading_constant}. It is somewhat subtle as it involves taking a sum over integrals with respect to the Brauer--Manin pairing and Tamagawa measure, since the resulting Brauer group can be infinite (modulo constants).

\subsection{Set-up} \label{sec:set-up_conjecture}

We follow closely the set-up from \cite[Section~3.3]{PSTVA19}. Let $X$ be a smooth projective variety over a number field $F$ and let $(X,D_\epsilon)$ be a smooth klt Campana orbifold with $D_\epsilon = \sum_{\alpha \in \mathcal{A}} \epsilon_\alpha D_\alpha$, for some $D_\alpha$.
For each $\alpha \in \mathcal A$, we fix an adelic metrization for $\mathcal O(D_\alpha)$ and denote its adelically metrized divisor as $\mathcal D_\alpha$.
Let $\mathcal{L}$ be an adelically metrised log adjoint rigid big and nef divisor so that $a((X, D_\epsilon), L) L + K_X + D_\epsilon$ is $\mathbb Q$-linearly equivalent to an effective $\mathbb Q$-divisor $E$ on $X$ which has Iitaka dimension $0$. We assume that the support of $E$ is contained in the support of $D$. We also fix the adelic metrization for the canonical divisor $K_X$ which is induced by adelic metrizations of $\mathcal L$ and $\mathcal D_\alpha$. Our assumptions imply that if we let $U = X \setminus \mathrm{Supp}(E)$ denote the complement of the support of $E$ in $X$, then we have $k[U]^\times = k^\times$. Indeed, if there is a non-constant and non-vanishing regular function on $U$, then it defines a non-trivial linear equivalence relation among irreducible components of $E$. However, this contradicts the fact that the Iitaka dimension of $E$ is $0$. %

Let $X^{\circ} = X \setminus \mathrm{Supp}(D_\epsilon)$. The space $U(F_v)$ can be endowed with a local Tamagawa measure $\tau_{U, v}$ by \cite[Sec.~2.1.8]{C-T-integral} for any $v \in \Omega_F$.
The associated local Campana Tamagawa measure is $\mathsf H_v(D_\epsilon, \cdot) \tau_{U, v}$. Let $\lambda_v = \mathsf L_v(\Pic(\overline{U}), 1)$ be the convergence factor at $v$ as in \cite[Definition 2.2]{C-T-integral}. We define $\tau_{U, D_\epsilon} = \mathsf L^*(\Pic(\overline{U}), 1)\prod_{v}\lambda_v^{-1}\mathsf H_v(D_\epsilon, \cdot) \tau_{U, v}$, where $\mathsf L^*(\Pic(\overline{U}), 1)$ is the leading coefficient in the Laurent series expansion of $\mathsf L(\Pic(\overline{U}), s)$ at $s = 1$.

\subsection{Campana Brauer group}

\begin{defn}\label{def:Br_Campana}
We define the \emph{Campana Brauer group} of $(X,D_\epsilon,L)$ to be
\[
\Br((X,D_\epsilon), L) = \{ \beta \in \Br X^{\circ} : m_\alpha \res_{D_\alpha}(\beta) = 0 \,\, \forall D_\alpha \subseteq U \setminus X^\circ\}.
\]
\end{defn}
Here $D_\alpha$ is a divisor occurring as an irreducible component of $U \setminus X^\circ$ and 
$\res_{D}: \Br X^{\circ} \to \H^1(\kappa(D),\Q/\Z)$ denotes the residue map at $D$.
This group (without reference to $L$) first appeared in the context of semi-integral points in \cite[Def.~3.14]{MNS22}. We define the algebraic part by $\Br_1((X,D_\epsilon), L) = \Br((X,D_\epsilon), L) \cap \Br_1 X^{\circ}$. A particularly important case is the log anticanonical divisor $\Br((X,D_\epsilon), -K_X - D_\epsilon)$, where one takes all boundary divisors $D_\epsilon$ in the definition since $U = X$; we call this simply the \emph{Campana Brauer group} $\Br(X,D_\epsilon)$ of the pair $(X,D_\epsilon)$. The group $\Br(X,D_\epsilon)/\Br \Q$ can be infinite  even if $(X,D_\epsilon)$ is log Fano. This complicates matters greatly and reflects the fact that in many examples, the orbifold fundamental group of $(X,D_{\epsilon})$, in the sense of \cite[Def.~11.1]{Cam11}, can be non-trivial.

\begin{exam} \label{ex:Gm}
	Consider the pair $(\mathbb{P}^1, (1/2)0 + (1/2)\infty)$ over $\mathbb{Q}$.
	Then the Campana Brauer group is generated by $\Br \mathbb{Q}$ 
	and all the quaternion algebras
	$(t,a)$ where $t$ is a coordinate function and $a \in \Q^\times$. 
	These are all distinct
	as $a$ ranges over representatives of $\Q^{\times}/\Q^{\times 2}$.
\end{exam}

Our conjecture features a quite delicate constant coming from taking the Brauer-Manin obstruction with respect to  a possibly infinite Brauer group. To interpret such a result we view it as sum of integrals over local invariants. This approach is inspired by recent perspectives on the volume of the Brauer--Manin set for local solubility in families \cite[Sec.~3.7.2]{LRS22}, integral Manin's conjecture \cite[Thm.~6.11]{Santens}, and Malle's conjecture \cite[Lem.~7.18]{LS24}.
For $b \in \Br((X,D_\epsilon), L)$ we define
\begin{align} \label{def:Brauer_integral}
&\widehat{\tau}_{U, D_\epsilon}(b) 
=\mathsf L^*(\Pic(\overline{U}), 1)\prod_{v \in \Omega_F}\lambda_v^{-1}\int_{x_v \in X^\circ(F_v)_\epsilon} \frac{e^{2 \pi i \inv_v b(x_v)} \mathsf H_v(D_\epsilon, x_v) \mathrm{d}\tau_{U,v}}
{\mathsf H_v(a((X, D_\epsilon), L)L + K_X + D_\epsilon, x_v)}.
\end{align}
Each local integral here depends on $b$, but the product only depends on the class of $b$ modulo $\Br F$. 
	

\subsection{Conjecture}

\begin{conj} \label{conj:leading_constant}
The conjectural leading constant in the asymptotic formula for the number of Campana points of bounded $\mathcal{L}$-height is
\begin{align}
\label{eq:conj_B_0}
c(F,S,(\mathcal{X},\mathcal{D}_\epsilon),\mathcal{L}) = 
\frac{\alpha( (X,D_\epsilon), L )}{a((X, D_\epsilon), L)(b(F, (X, D_\epsilon), L) - 1)!} \sum_{b \in \Br_1((X,D_\epsilon), L)/\Br F} \widehat{\tau}_{U, D_\epsilon}(b).
\end{align}
\end{conj}
 Here $a((X, D_\epsilon), L)$ and $b(F, (X, D_\epsilon), L)$ are as in \cite[Conj.~1.1]{PSTVA19}, and we recall from \cite[Sec.~3.3]{PSTVA19} that
\begin{equation} \label{def:alpha}
	\alpha( (X,D_\epsilon), L ) =\int_{\Lambda^*}  e^{-\langle L, \mathbf{x}\rangle} \mathrm{d}\mathbf{x} \cdot  \prod_{\alpha \in \mathcal{A}(L)}(1- \epsilon_\alpha),
\end{equation}
where $\Lambda^*$ denotes the dual of the image of the effective cone of divisors of $X$ under the map $\Pic X \to \Pic U$, and $\mathcal{A}(L)$ denotes the boundary divisors which are not contained in the support of $E$.

The new part of the conjecture is the sum over the Campana Brauer group. This should be interpreted as a substitute for the adelic volume of the orthogonal complement to $\Br_1((X,D_\epsilon), L)$ weighted by the cardinality of $|\Br_1((X,D_\epsilon), L)/\Br F|$ (which may be infinite!). We do not try to make this rigorous, but it is inspired by the character orthogonality arguments from \cite[Lem.~3.13]{LRS22}, \cite[Thm.~6.11]{Santens}, and  \cite[Lem.~7.18]{LS24}.

In practice the expression \eqref{eq:conj_B_0} actually seems to be easier to work with than a volume: it is given by a sum of explicit Euler products, whereas the volume can be difficult to calculate.  However we do not know whether this sum exists in general (naturally we conjecture that the sum over $b$ is convergent). Moreover it is not clear that is non-zero; a method of proof of non-vanishing can be extracted from the proof of Theorem~\ref{thm:leading_constant}. In some cases the sum reduces to a finite sum; however we give an explicit example in Section \ref{sec:sum_squareful} where the sum is genuinely infinite. 

\begin{rem}
	Conjecture \ref{conj:leading_constant} is stated with two caveats. 
	Firstly, we have restricted to algebraic Brauer group elements only
	as this is compatible with the current list of known examples. The r\^{o}le
	of transcendental Brauer group elements in Manin's conjecture and other counting problems
	is currently fairly mysterious; this should be investigated
	further.

	Secondly, the examples we consider below all have the property that the orbifold fundamental
	group, in the sense of \cite[Def.~11.1]{Cam11}, is abelian. 
	It may be  possible that Conjecture~\ref{conj:leading_constant} should take a different
	form when the orbifold fundamental group is non-abelian. We are grateful to Tim Santens
	for this observation.
\end{rem}

\subsection{Verification for the wonderful compactification}
We now verify Conjecture \ref{conj:leading_constant} in our case of the wonderful compactification. We take $X$ to be a wonderful compactification of a semi-simple adjoint algebraic group $\G$ and $L = \sum_{\alpha} \lambda_\alpha D_\alpha$ an adjoint rigid big line bundle on $X$. We start with the effective cone constant. Recall that $\mathcal{A}(L)$ denotes the boundary divisors which are not contained in the support of $E$. Note that this is compatible with the notation in Section~\ref{sect:poles}. Indeed, using the notation there, we have
\[
D_\alpha \subseteq \mathrm{Supp}(E) \iff \frac{\kappa_\alpha+1-\epsilon_\alpha}{\lambda_\alpha} > a.
\]
Thus our claim follows.

\begin{lemma} \label{lem:effective_cone}
	\[
	\alpha( (X,D_\epsilon), L )= \frac{1}{|\Pic (\G)|} \prod_{\alpha \in \mathcal{A}(L)}\frac{1}{m_\alpha \lambda_\alpha}.
	\]
\end{lemma}
\begin{proof}
	Let $M$ be the lattice generated by the components in $\mathcal{A}(L)$.
	We have  the exact sequence
	\[
	0 \to M\to \Pic(U) \to \Pic(\G) \to 0.
	\]
	Let $\Lambda \subset M$ be the simplicial cone generated by components in $\mathcal{A}(L)$.
	Then the definition of \cite[Section 3.3]{PSTVA19} gives us 
	\[
	\alpha( (X,D_\epsilon), L ) = \frac{1}{|\Pic (\G)|}  \chi_{\Lambda}(\rho(L)) \prod_{\alpha \in \mathcal{A}(L)}\frac{1}{m_\alpha} = \frac{1}{|\Pic (\G)|}  \prod_{\alpha \in \mathcal{A}(L)}\frac{1}{m_\alpha \lambda_\alpha}. \qedhere
	\]
\end{proof}

We utilise the correspondence between Brauer group elements and automorphic characters proved in
\cite{LTBT18}. Specifically \cite[Thm.~2.9, Cor.~2.11]{LTBT18} implies that for each 
$b \in \Br_1 \G$ the map induced by the Brauer--Manin pairing
$$\chi_b: \G(\Adele_F) \to S^1, \quad (\gamma_v) \mapsto e^{2 \pi i \sum_v \inv_v b(\gamma_v)}$$
is an automorphic character (here we view $\Q/\Z \subset S^1$ via the map $x \mapsto e^{2 \pi i x}$), and that this induces an isomorphism
\begin{equation} \label{eqn:Br_iso}
	\Br_1 \G/ \Br F \to (\G(\Adele_F)/ \G(F) )^\wedge.
\end{equation}
We next verify the conjecture in our case. Moreover we show that the sum in \eqref{eq:conj_B_0} reduces to a finite sum (despite the corresponding Brauer group being infinite modulo constants).

\begin{thm} \label{thm:conjecture_true}
	The leading constant in Theorem~\ref{thm:leading_constant} equals the conjectural constant
	in Conjecture \ref{conj:leading_constant}. Moreover let
	$\Br_1((X,D_\epsilon), L)^{\K_0}$ denote the subset of 
	$\Br_1((X,D_\epsilon), L)$ orthogonal to $\K_0$ with respect to the Brauer--Manin pairing.
	Then $\Br_1((X,D_\epsilon), L)^{\K_0}/\Br F$ is finite and
	$$\sum_{b \in \Br_1((X,D_\epsilon), L)/\Br F} \widehat{\tau}_{U, D_\epsilon}(b)
	= \sum_{b \in \Br_1((X,D_\epsilon), L)^{\K_0}/\Br F} \widehat{\tau}_{U, D_\epsilon}(b).$$
\end{thm}
\begin{proof}
We take as our starting point the formula
\[
	c = \lim_{s \to a} \frac{(s-a)^b}{a(b-1)!}| \mathscr{X}_{D_{\epsilon}, L} (\G)| \int_{\G(\mathbb A_F)^{ \mathscr{X}_{D_{\epsilon}, L} (\G) }}\delta_\epsilon(g)\mathsf H(sL, g)^{-1} \mathrm dg
\]
from Theorem \ref{thm:leading_constant}. 
We first study the Campana Brauer group in Conjecture \ref{conj:leading_constant}, using \eqref{eqn:Br_iso}. Recall that we denote by $\mathscr{X}$ the collection of all automorphic characters of $\G$. Then we claim that
\begin{align}\label{eqn:limit_characters}
\begin{split}
& \sum_{\chi \in \mathscr{X}} \lim_{s \to a} (s-a)^b  \int_{\G(\mathbb A_F)}\delta_\epsilon(g) \mathsf H(sL, g)^{-1} \chi(g) \mathrm dg  \\
= & \sum_{\chi \in \mathscr{X}_{D_{\epsilon}, L} (\G)} \lim_{s \to a} (s-a)^b  \int_{\G(\mathbb A_F)}\delta_\epsilon(g) \mathsf H(sL, g)^{-1} \chi(g) \mathrm dg.
\end{split}
\end{align}
Indeed: this follows from character orthogonality and the fact that only the characters in $\mathscr{X}_{D_{\epsilon}, L}(\G)$ can contribute to the leading pole. However, the residue of an element of $\Br_1 \G$ along $D_\alpha$
corresponds (up to sign) to the induced Hecke character $\chi_\alpha$ 
\cite[Lem.~4.16]{LTBT18}. Therefore we deduce that the map
$\Br_1((X,D_\epsilon), L)^{\K_0} \to ( (\G(\Adele_F)/ \G(F) )^\sim$
induced by the Brauer--Manin pairing has image 
$\mathscr{X}_{D_{\epsilon}, L_\lambda} (\G)$ and kernel $\Br F$. This interprets our integral in terms of the Campana Brauer group and shows that the stated finitely many terms contribute. It thus suffices to express the integral in terms of the Tamagawa measure and identify the effective cone constant.

Let $\chi \in \mathscr{X}$. We first consider convergence factors. These come from the Artin $L$-function $L(\Pic \bar{U},s)$. However in our case $\Pic(\bar{U}) \otimes \Q$ is the vector space generated by the boundary divisors in $\mathcal{A}(L)$. The corresponding Artin $L$-function is thus
$\prod_{\alpha \in \mathcal{A}(L)} \zeta_{F_\alpha}(s).$
We introduce the local Euler factors directly into the integral, which shows that 
\begin{align*}
&  \lim_{s \to a} (s-a)^b  \int_{\G(\mathbb A_F)} \delta_\epsilon(g) \mathsf H(sL, g)^{-1} \chi(g) \mathrm dg \\
= & \lim_{s \to a} (s-a)^b \int_{\G(\mathbb A_F)} 
\prod_{\substack{\alpha \in \mathcal{A}(L) \\ v \in \Omega_F }} \frac{\zeta_{F_\alpha,v}(m_\alpha(s\lambda_\alpha - \kappa_\alpha))}{\zeta_{F_\alpha,v}(m_\alpha(s\lambda_\alpha - \kappa_\alpha))}\delta_\epsilon(g)\mathsf H(sL, g)^{-1} \chi(g) \mathrm dg \\
=& \lim_{s \to a} (s-a)^b \prod_{\alpha \in \mathcal A(L)} \zeta_{F_\alpha}(m_\alpha(s\lambda_\alpha-\kappa_\alpha )) \\
&\times  \lim_{s \to a}  \int_{\G(\mathbb A_F)} 
\prod_{\substack{\alpha \in \mathcal{A}(L) \\ v \in \Omega_F }} \frac{1}{\zeta_{F_\alpha,v}(m_\alpha(s\lambda_\alpha - \kappa_\alpha))}\delta_\epsilon(g)\mathsf H(sL, g)^{-1} \chi(g) \mathrm dg\\
= &  \prod_{\alpha \in \mathcal{A}(L)}\frac{\zeta^*_{F_{\alpha}}(1)}{m_\alpha \lambda_\alpha} 
\lim_{s \to a} \int_{\G(\mathbb A_F)} 
\prod_{\substack{\alpha \in \mathcal{A}(L) \\ v \in \Omega_F}} \frac{1}{\zeta_{F_\alpha,v}(m_\alpha(s\lambda_\alpha - \kappa_\alpha))}\delta_\epsilon(g)\mathsf H(sL, g)^{-1} \chi(g) \mathrm dg,
\end{align*}
where we use the fact that $\lim_{s \to a}(s-a) \zeta_{F_\alpha}(m_\alpha(\lambda_\alpha  s - \kappa_\alpha)) = \frac{1}{\lambda_\alpha m_\alpha}  \zeta^*_{F_\alpha}(1)$.
We next take the limit over $s$ inside the integral. This is justified by 
Theorem \ref{thm:one-dim}, which shows that the above is given by an absolutely convergent Euler product on $\Re s > a - \varepsilon$ for some $\varepsilon > 0$. The change of limit is now justified by applying the dominated convergence theorem, first to the logarithm of the product to justify taking the limit inside the product over all $v$, then again to justify taking the limit inside the integral over $G(F_v)$.
We find that the above  equals
\begin{equation} \label{eqn:convergence_factors}
\prod_{\alpha \in \mathcal{A}(L)}\frac{\zeta^*_{F_\alpha}(1)}{m_\alpha \lambda_\alpha}  \int_{\prod_{v}\G(F_v)}\prod_{\substack{\alpha \in \mathcal{A}(L) \\ v \in \Omega_F }} \zeta_{F_\alpha,v}(1)^{-1}\delta_\epsilon(g)\mathsf H(aL, g)^{-1} \chi(g) \mathrm dg.
\end{equation}
Regarding the measures, recall that $\mathrm{d}g$ is a Haar measure on $\G(\Adele_F)$ normalised so that $\vol(\G(\Adele_F)/\G(F))=1$. Let $\omega$ be a left invariant top degree
differential form on $\G$. As $\G$ is semi-simple, the Haar measure $|\omega| = \prod_v|\omega|_v$ is the classical (i.e.~Weil's) Tamagawa
measure on $\G(\Adele_F)$. Denote by $\omega(\G)$ the 
Tamagawa number of $\G$, i.e.~the measure of $\G(\Adele_F)/\G(F)$ with respect to $|\omega|$.
A theorem of Ono \cite[Thm.~10.1]{sansuc} states that
$\omega(\G) = |\Pic \G|/|\Sha(\G)|.$
However, as $G$ is adjoint we have $\Sha(\G) = 0$ \cite[Cor.~5.4]{sansuc}. Thus $|\omega| = |\Pic G|\mathrm{d}g$. Using the relation $\tau_{U,v} = |\omega|/\|\omega\|$ (see \cite[Sec.~2.1.10]{C-T-integral}) we conclude that
\begin{equation} \label{eqn:measures_Pic}
\prod_v \prod_{\substack{\alpha \in \mathcal{A}(L)}} \frac{\mathrm{d}g_v}{\zeta_{F_\alpha,v}(1)\mathsf H_v(aL, \cdot)} = \frac{1}{|\Pic \G|} \prod_v \frac{\mathsf H_v(D_\epsilon, \cdot) \tau_{U, v}}{\lambda_v \mathsf H_v(aL + K_X + D_\epsilon, \cdot)}.
\end{equation}
(Here $\mathsf H_v(D_\epsilon, \cdot) \tau_{U, v}$ is the local Campana Tamagawa measure and $\lambda_v$ the convergence factor, as in Section \ref{sec:set-up_conjecture}.) Combining this with Theorem \ref{thm:leading_constant}, Lemma \ref{lem:effective_cone}, \eqref {eqn:Br_iso}, \eqref{eqn:convergence_factors}, and \eqref{eqn:measures_Pic} completes the proof.
\end{proof}

\subsection{Other group compactifications}
Conjecture \ref{conj:leading_constant} is highly speculative with the intention of stimulating further research into the nature of the leading constant for Campana points. We have verified it in the case of the wonderful compactification.  We discuss here  compatibility with existing results in the literature, with an emphasis on the new aspect involving the Camapana Brauer group.

\subsubsection{Equivariant compactifications of $\mathbb{G}_a^n$}
This case was considered in the paper \cite{PSTVA19} where the original conjecture was presented. Here the Campana Brauer group plays no role, since the Brauer group of $\mathbb G_a^n$ is constant.

\subsubsection{Squareful values of binary quadratic forms}
This was the counter-example to the original conjecture, considered independently by Streeter \cite{Str22} and Shute \cite{Shute21b}. We focus on the paper of Streeter, since he also uses the height zeta function approach which is more easy to compare with our own. Since the harmonic analysis approach is used, it is very similar to our situation, hence we shall be brief.

In \cite[Lem.~6.10]{Str22} identifies the automorphic characters $\chi$ which can give rise to the leading singularity, namely they are exactly those which are invariant under a suitable compact-open subgroup (see \cite[Def.~4.17]{Str22}) and for which $\chi^m$ is trivial, where $m$ is the multiplicity of his unique boundary divisor. This is the exact analogue of the characters which appear in our Theorem \ref{thm:leading_constant}. A variant of the argument given in Theorem \ref{thm:conjecture_true} shows that Streeter's result is compatible with Conjecture \ref{conj:leading_constant}. (The correspondence between automorphic characters of algebraic tori and Brauer group elements in this case is obtained in \cite[Thm.~4.5]{Lou18}.)

\subsubsection{Split toric varieties}
An asymptotic was obtained here by Pieropan and Schindler \cite{PS20}. This case is subtle as it is compatible with both Conjecture \ref{conj:leading_constant} and the conjecture in \cite{PSTVA19}, despite the relevant Campana Brauer group being enormous (cf.~Example \ref{ex:Gm}). This is due to very particular features of the exact situation, coming from the precise choice of height and the fact that the ground field $\Q$ has trivial class group.

The result \cite[Thm.~1.2]{PS20} concerns the log anticanonical height and is proved using the hyperbola method and not the harmonic analysis approach. Still it is possible to analyse the leading constant using the harmonic analysis approach, since it can be interpreted via Fourier transforms of height functions. One can then perform an analysis similar to the proof of Theorems \ref{thm:leading_constant} and \ref{thm:conjecture_true}. One is led to consider a finite collection of automorphic characters of the form
$$
\mathscr X(\Gm^n) = \{ \chi \in \Hom( \Gm^n(\Adele_F)/\Gm^n(F), S^1):  \chi_\alpha^{m_\alpha} = 1 \text{ for all } \alpha \in \mathcal{A} \text{ and } \chi(\K) = 1\}
$$
as in Theorem \ref{thm:leading_constant}. Here $\mathcal{A}$ is the collection of all boundary divisors, $\chi_\alpha$ Hecke characters coming from the boundary, and $\K$ is the maximal compact open subgroup of $\Gm^n(\Adele_F)$, since their choice of height in \cite[Sec.~6.3]{PS20} is invariant under this subgroup.

We can give a description of such automorphic characters in terms of the class group of the number field which shows that, in the case $F = \Q$ considered in \cite[Thm.~1.2]{PS20}, this group is trivial, which implies that it is only the trivial Brauer group element which contributes to the sum \eqref{eq:conj_B_0}. Thus  \cite[Thm.~1.2]{PS20} is indeed compatible with Conjecture \ref{conj:leading_constant}.

\begin{lemma}
	Let $m = \mathrm{lcm}_\alpha m_\alpha$. Then $\mathscr X(\Gm^n)$ is isomorphic to a subgroup
	of $\mathrm{Cl}_F[m]^n$, where $\mathrm{Cl}_F$ denotes the class group of $F$. In particular
	$\mathscr X(\Gm^n)$ is trivial for $F = \Q$.
\end{lemma}
\begin{proof}
	Our conditions imply that any $\chi \in \mathscr X(\Gm^n)$ is $m$-torsion. To prove the result it suffices 
	to consider	the case $n =1$. Then we have
	$$\{ \chi \in \Hom( \Gm(\Adele_F)/\Gm(F), S^1):  \chi^{m} = 1 \text{ and } \chi(\K) = 1 \} \cong \mathrm{Cl}_F[m].$$
	Indeed, by global class field theory, an automorphic character $\chi$ of $F$ corresponds
	to a map $\Gal(\bar{F}/F) \to S^1$. If $\chi$ is everywhere unramified then $\chi$
	factors through the Hilbert class field, whose Galois group is precisely $\mathrm{Cl}_F$.
\end{proof}

\subsubsection{Non-split toric varieties}
Since the appearance of our work, Shute and Streeter \cite[Thm.~1.1]{SS24} have verified Conjecture \ref{conj:leading_constant} for arbitrary toric varieties.

\subsection{Sums of three squareful numbers} \label{sec:sum_squareful}
The equation $z_0 + z_1 = z_2$ with each variable squarefull was studied in \cite{BVV12}. A conjecture for the leading constant was put forward in \cite[Sec.~2.4]{BVV12} as an infinite sum of Euler products, which follows from writing the collection of Campana points as an infinite disjoint union of thin sets given by images of rational points from diagonal conics. This conjecture was shown to be incompatible with the conjecture in \cite[Section~3.3]{PSTVA19} by Shute \cite[Thm.~1.2]{Shute21b}.  We now provide evidence that this is actually compatible with Conjecture~\ref{conj:leading_constant}.

In what follows we want to focus on explaining the new phenomenon, namely the sum over the Campana Brauer group, and do not want to get distracted by technical details. In particular we only provide a sketch of a proof by ignoring archimedean and $2$-adic considerations. We also consider instead the equation
\begin{equation}
	Z: \qquad z_0 + z_1 + z_2 =0 \label{eqn:symmteric}
\end{equation}
as it leads to more symmetric formulae. We take the three points $P_i: z_i = 0$  with multiplicity~$2$. The height function we use is $H(z) = \prod_v \max\{ |z_0|_v, |z_1|_v, |z_0 + z_1|_v\}$. An interesting feature of this example is that the sum \eqref{eq:conj_B_0} is genuinely infinite, unlike the case of wonderful compactifications and toric varieties. For this reason it looks difficult to interpret as the measure of some Brauer--Manin set with respect to infinitely many Brauer group elements, hence why we have left \eqref{eq:conj_B_0} as simply a sum. This also presumably reflects the difficulty of proving an asymptotic formula in this case. We work over $\Q$.

\subsubsection{Campana Brauer group}
We first calculate the Campana Brauer group in our setting. We have $D_\epsilon = (1/2)P_0 + (1/2)P_1 + (1/2)P_2$.

\begin{lemma} \label{lem:d_i}
	The map
	\begin{align*}\{(d_0,d_1,d_2) \in \Z^3 : \mu^2(d_0d_1d_2) = 1, \sign(d_0d_1d_2) = 1\}
	&\to \Br((Z,D_\epsilon), L) /\Br \Q \\
	(d_0,d_1,d_2) &\mapsto (d_0,z_1/z_2) + (d_1,z_0/z_2) + (d_2,z_0/z_1)
	\end{align*}
	is a bijection. In particular $\Br((Z,D_\epsilon), L) /\Br \Q
	\cong \H^1(\Q, \Z/2\Z \times \Z/2\Z)$. 
\end{lemma}
\begin{proof}
	First it is clearly well-defined. For injectivity, it suffices to show that for $(d_0,d_1,d_2)$
	if 
	$$(d_0,z_1/z_2) + (d_1,z_0/z_2) + (d_2,z_0/z_1) \in \Br \Q$$
	then $d_0 = d_1 = d_2 = 1$. However the residue of this element along $P_i$
	is $d_jd_k$ where $\{i,j,k\} = \{0,1,2\}$. As $d_j$ and $d_k$ are coprime and squarefree,
	if this residue is trivial then we must have $|d_j| = |d_k| = 1$. In which case
	the system of equations 
	$$d_0d_1 = d_1d_2 = d_0d_2 = d_0d_1d_2 = 1$$
	is easily seen to imply that $d_0 = d_1 = d_2 = 1$, as required.
	
	For surjectivity, let $b$ lie in the Campana Brauer group. The residue at each $P_i$
	is two torsion. Denote the corresponding elements by $c_i \in \Q^{\times}/\Q^{\times 2}$.
	The Faddeev exact sequence \cite[Thm.~1.5.2]{Brauer}
	implies that $c_0c_1c_2$ is a square. In particular without loss of generality we may assume
	that $c_2 = c_0c_1/\gcd(c_0,c_1)$. In which case one checks that
	$$b + (\sign(c_1c_2)\gcd(c_1,c_2),z_1/z_2) + (\sign(c_0c_2)\gcd(c_0,c_2),z_0/z_2) + (\sign(c_0c_1)\gcd(c_0,c_1),z_0/z_1)$$
	has trivial residues, thus is constant.
\end{proof}

\subsubsection{Local densities}

We next calculate the Brauer group integrals from \eqref{def:Brauer_integral}. For this we require the following (in the statement $(d,t)_p$ denotes the Hilbert symbol over $\Q_p$).

\begin{lemma}  \label{lem:Hilbert_symbol_integral}
	Let $p$ be an odd prime, $n \in \Z$ and $d \in \Z_p$ with $v_p(d) \leq 1$. We have
	$$\int_{\substack{t \in \Q_p \\ v_p(t) = n}} (d,t)_p \mathrm{d} t\\
	=
	\begin{cases}
	0, \quad & p \mid d, \\
	\frac{(1 - 1/p)}{p^n}\left(\frac{d}{p}\right)^n, \quad & p \nmid d.	
	\end{cases}
	$$
\end{lemma}
\begin{proof}
	First assume that $p \mid d$. Take $u \in \Z_p^\times$ such that 
	$(d,u)_p = -1$, i.e.~$u$ is a quadratic non-residue modulo $p$.
	Making the change of variables $t \mapsto tu$ gives
	$$\int_{\substack{t \in \Q_p \\ v_p(t) = n}} (d,t)_p \mathrm{d} t
	= \int_{\substack{t \in \Q_p \\ v_p(t) = n}} (d,tu)_p |u|_p \mathrm{d} t
	= (d,u)_p \int_{\substack{t \in \Q_p \\ v_p(t) = n}} (d,t)_p \mathrm{d} t$$
	therefore the integral is $0$. Now assume that $p \nmid d$. Then 
	$(d,t)_p = \left(\frac{d}{p}\right)^{v_p(t)}$. Using that the region of integration has measure
	$p^{-n} - p^{-n-1}$ completes the proof.
\end{proof}

We will also require the following for a Legendre symbol sum.

\begin{lemma} \label{lem:quadratic_legendre_sum}
	Let $p$ be an odd prime. Then
	$$\sum_{a \in \F_p} \left( \frac{a(1 +a)}{p} \right) = -1.$$
\end{lemma}
\begin{proof}
	One easily checks that 
	$$\#\{(a,c) \in \F_p^2 : a^2 + a = c^2\} = 
	\sum_{a \in \F_p}\left(1 +  \left( \frac{a(1 +a)}{p} \right) \right).$$
	However this is the number of $\F_p$-points on a smooth affine conic with two rational points
	at infinity, which is thus $p-1$. The result follows.
\end{proof}

We now calculate the Brauer integrals from \eqref{def:Brauer_integral} at odd primes.

\begin{lemma} \label{lem:Brauer_integral}
	Let $b = (d_0,z_1/z_2) + (d_1,z_0/z_2) + (d_2,z_0/z_1)$ as in Lemma \ref{lem:d_i}.
	Let $p$ be an odd prime. Then
	\begin{align*}
	&\int_{Z^{\circ}(\mathbb{Q}_p)_\epsilon} e^{2 \pi i \inv_p b(z_p)} H_p(D_\epsilon, z_p) 
	\mathrm{d}\tau_{Z,p} \\
	&=
	\begin{cases}
		\left(\left(\frac{-d_0d_1}{p}\right) + \left(\frac{-d_0d_2}{p}\right)
		+  \left(\frac{-d_1d_2}{p}\right)\right) p^{-3/2},\quad &p \mid d_0d_1d_2, \\
		 1 + p^{-1} + 
		\left(\left(\frac{d_0d_1}{p}\right) + \left(\frac{d_0d_2}{p}\right)+ 
		\left(\frac{d_1d_2}{p}\right)\right)p^{-3/2},
		\quad &p \nmid d_0d_1d_2.
	\end{cases}
	\end{align*}
\end{lemma}
\begin{proof}
	The first step of our proof is the same as in the proof of 
	\cite[Lem.~3.1]{Shute21b}. We use the parametrisation
	$t \mapsto (t : 1 : -1 - t)$. With respect to this, the $3$ marked points are $
	t = 0, -1$ and $\infty$. As in \cite[(3.10)]{Shute21b} we have that the integral equals
	\begin{equation} \label{eqn:Shute_integral}
	\int_{\substack{ t \in \Q_p \\ v_p(t), v_p(1+t) \neq \pm 1}}
	\frac{(d_0, -(1 + t))_p\cdot (d_1,-t(1+t))_p\cdot (d_2,t)_p \mathrm{d} t}
	{|t(1+t)|^{1/2}_p \max\{|t|_p,1,|1+t|_p\}^{1/2}}.
	\end{equation}
	
	\noindent \textbf{Avoiding the marked points:}
	We first consider the case where $v_p(t) = v_p(1 + t) = 0$, so the reduction modulo $p$ does 
	not meet the marked points. Here the integral \eqref{eqn:Shute_integral} becomes
	\begin{align*}
		& \sum_{\substack{ a \in \mathbb{F}_p \\ a \neq 0,-1}}
		\int_{\substack{ t \in \Z_p \\ t \equiv a \bmod p}}
		(d_0, -(1 + t))_p\cdot (d_1,-t(1+t))_p\cdot (d_2,t)_p \mathrm{d} t \\
		& = p^{-1}\sum_{\substack{ a \in \mathbb{F}_p \\ a \neq 0,-1}}
		\left(\frac{-(1+a)}{p^{v_p(d_0)}}\right)
		\left(\frac{-a(1+a)}{p^{v_p(d_1)}}\right)		
		\left(\frac{a}{p^{v_p(d_2)}}\right)
	\end{align*}
	where the symbol is the Legendre symbol.
	If $p \nmid d_0d_1d_2$ this equals $1 - 2/p$.

	If $p \mid d_0$ then we obtain
	$$p^{-1}\sum_{\substack{ a \in \mathbb{F}_p \\ a \neq 0,-1}}
	\left(\frac{-(1+a)}{p}\right)
	= p^{-1}\sum_{\substack{ b \in \mathbb{F}_p \\ b \neq -1}}
	\left(\frac{b}{p}\right)  = -p^{-1}\left(\frac{-1}{p}\right).$$
	
	If $p \mid d_1$ then we obtain
	$$p^{-1}\sum_{\substack{ a \in \mathbb{F}_p \\ a \neq 0,-1}}
	\left(\frac{-a(1+a)}{p}\right) = 
	-p^{-1}\left(\frac{-1}{p}\right)$$	
	by Lemma \ref{lem:quadratic_legendre_sum}.

	If $p \mid d_2$ then we obtain
	$$p^{-1}\sum_{\substack{ a \in \mathbb{F}_p \\ a \neq 0,-1}}
	\left(\frac{a}{p}\right) = -p^{-1}\left(\frac{-1}{p}\right).$$	
	We next focus on the case where we are $p$-adically close to one
	of the marked points.
	
	\noindent \textbf{Near to $0$:} 
	Firstly assume that $v_p(t) \geq 2$ so that we are close to $0$.
	Here $|1 + t|_p = \max\{|t|_p,1|\} = 1$ and $(d_i,-(1+t))_p = (d_i,-1)_p$.
	Hence \eqref{eqn:Shute_integral} becomes
	\begin{equation} \label{eqn:tau_0}
	\int_{\substack{ t \in \Q_p \\ v_p(t) \geq 2}}
	\frac{(d_0, -(1 + t))_p\cdot (d_1,-t(1+t))_p\cdot (d_2,t)_p}{|t|^{1/2}_p} \mathrm{d} t
	 = (d_0d_1,-1)_p\int_{\substack{ t \in \Q_p \\ v_p(t) \geq 2}}
	\frac{(d_1d_2,t)_p}{|t|^{1/2}_p} \mathrm{d} t.
	\end{equation}
	By Lemma \ref{lem:Hilbert_symbol_integral} this is $0$ for $p \mid d_1d_2$.
	If $p \nmid d_1d_2$ then $(d_1,-1)_p = 1$, hence we obtain that 
	\eqref{eqn:tau_0} equals $(d_0,-1)_p$ times
	$$\sum_{j = 2}^\infty \left(\frac{d_1d_2}{p}\right)^j(1 - p^{-1})p^{-j/2}
	= (1 - p^{-1}) \frac{p^{-1}}{ 1 - \left(\frac{d_1d_2}{p}\right) p^{-1/2} }
	= p^{-1} + \left(\frac{d_1d_2}{p}\right)p^{-3/2}.$$
	
	\noindent \textbf{Near to $\infty$:} 
	Next when we are $p$-adically close to $\infty$, i.e.~$v_p(t) \leq -2$. Here we have
	$$\int_{\substack{ t \in \Q_p \\ v_p(t) \leq -2}}
	\frac{(d_0, -(1 + t))_p\cdot (d_1,-t(1+t))_p\cdot (d_2,t)_p}{|t|^{3/2}_p} \mathrm{d} t.$$	
	We make the change
	of variables $u=1/t$ to obtain
	$$\int_{\substack{ u \in \Q_p \\ v_p(u) \geq 2}}
	\frac{(d_0, -u(1 + u))_p\cdot (d_1,-(u+1))_p\cdot (d_2,u)_p}{|u|^{1/2}_p} \mathrm{d} u
	 = (d_0d_1,-1)_p\int_{\substack{ u \in \Q_p \\ v_p(u) \geq 2}}
	\frac{(d_0d_2, u)_p}{|u|^{1/2}_p} \mathrm{d} u$$
	which, as in \eqref{eqn:tau_0}, is $0$ for $p \mid d_0d_2$ and 
	$(d_1,-1)_p(p^{-1} + \left(\frac{d_0d_2}{p}\right)p^{-3/2})$
	otherwise.
	
	\noindent \textbf{Near to $-1$:} 
	Finally close to $-1$, we have the integral
	$$\int_{\substack{ t \in \Q_p \\ v_p(t +1) \geq 2}}
	\frac{(d_0, -(1 + t))_p\cdot (d_1,-t(1+t))_p\cdot (d_2,t)_p}{|1+t|^{1/2}_p} \mathrm{d} t.$$	
	We make the change of variables $u=-(1+t)$ to obtain
	$$\int_{\substack{ u \in \Q_p \\ v_p(u) \geq 2}}
	\frac{(d_0, u)_p\cdot (d_1,-u(1+u))_p\cdot (d_2,-(1+u))_p}{|u|^{1/2}_p} \mathrm{d} u.$$
	As in \eqref{eqn:tau_0}, this equals $0$ for $p \mid d_0d_1$ and 
	$(d_2,-1)_p(p^{-1} + \left(\frac{d_0d_1}{p}\right)p^{-3/2})$
	otherwise.
	Combining all cases together gives the statement of the lemma.
\end{proof}

The convergence factors comes from $\zeta(s)$. Therefore, up to archimedean and $2$-adic contributions, we deduce that the Brauer group sum in \eqref{eq:conj_B_0} is given by
\begin{align} \label{conj:sum_three_squareful}
	&\sum_{d_0,d_1,d_2}\mu^2(d_0d_1d_2)
	\prod_{p \mid d_0}\frac{(1-p^{-1})\left(\frac{-d_0d_1}{p}\right)}{p^{3/2}} 
	\prod_{p \mid d_1}\frac{(1-p^{-1})\left(\frac{-d_0d_2}{p}\right)}{p^{3/2}} 
	\prod_{p \mid d_2}\frac{(1-p^{-1})\left(\frac{-d_1d_2}{p}\right)}{p^{3/2}}  \nonumber \\
	\times &\prod_{p \nmid d_0d_1d_2}(1-p^{-1})\left( 1 + \frac{1}{p} + \frac{\left(\frac{d_0d_1}{p}\right) + \left(\frac{d_0d_2}{p}\right) + \left(\frac{d_1d_2}{p}\right)}{p^{3/2}}\right).
\end{align}
We emphasise that this is absolutely convergent, with the absolute convergence coming from the Brauer integrals  at the bad primes.

\subsubsection{Compatibility with Conjecture \ref{conj:leading_constant}}
We now consider the conjecture from \cite{BVV12}, following Shute's paper \cite{Shute21b}. From \cite[Conj.~1.1, (4.12)]{Shute21b} we obtain the prediction
\begin{align*}
\frac{3}{\pi} \sum_{\mathbf{y} \in \Z^3}\frac{\mu^2(y_0 y_1 y_2) \gamma(y_0y_1y_2)}{(y_0y_1y_2)^{3/2}} \sigma_{2,\y} \varrho(\y)
\end{align*}
where
$$\gamma(n) = \prod_{\substack{p \mid n \\ p > 2}} \left(1 + \frac{1}{p}\right)^{-1}$$
and 
$$\varrho(\y) = \prod_{\substack{p \mid y_0 \\ p > 2}}\left( 1 + \left(\frac{-y_1y_2}{p}\right)\right) \prod_{\substack{p \mid y_1 \\ p > 2}}\left( 1 + \left(\frac{-y_0y_2}{p}\right)\right) \prod_{\substack{p \mid y_2 \\ p > 2}}\left( 1 + \left(\frac{-y_0y_1}{p}\right)\right).$$
(The factor $3$ comes from \cite[Lem.~4.1]{Shute21b} and we have different signs to Shute as we are using the more symmetric version \eqref{eqn:symmteric} of the equation.)
The term $\sigma_{2,\y}$ is a $2$-adic condition which can be found in \cite[(4.11)]{Shute21b}. We shall ignore this term as already explained. For $n,m \in \Z$ odd, squarefree, and coprime, we use the relation
$$\prod_{p \mid n} \left(1 + \left(\frac{m}{p}\right)\right) = \sum_{d \mid n}\left(\frac{m}{d}\right),$$
where the symbol is the Kronecker symbol. Applying this and changing the order of summation shows that the leading constant is given by
\begin{align*}
	\frac{3}{\pi}\sum_{\mathbf{d}}\sum_{\substack{\mathbf{y} \\ d_i \mid y_i}}\frac{\mu^2(y_0 y_1 y_2) \gamma(y_0y_1y_2)}{(y_0y_1y_2)^{3/2}} \left(\frac{-y_1y_2}{d_0}\right) \left(\frac{-y_0y_2}{d_1}\right) \left(\frac{-y_0y_1}{d_2}\right).
\end{align*}
We make the change of variables $y_i \mapsto d_iy_i$ and rearrange symbols to obtain
\begin{align*}
	&\frac{3}{\pi}\sum_{\mathbf{d}}\frac{\mu^2(d_0d_1d_2)\gamma(d_0d_1d_2)}{ (d_0d_1d_2)^{3/2}} \left(\frac{-d_1d_2}{d_0}\right) \left(\frac{-d_0d_2}{d_1}\right) \left(\frac{-d_0d_1}{d_2}\right) \\
	&\times \sum_{\substack{\mathbf{y}}}\frac{\mu^2(d_0d_1d_2y_0 y_1 y_2) \gamma(y_0y_1y_2)}{(y_0y_1y_2)^{3/2}}  \left(\frac{y_1y_2}{d_0}\right) \left(\frac{y_0y_2}{d_1}\right) \left(\frac{y_0y_1}{d_2}\right).
\end{align*}
For the outer sum we have
\begin{align}
\begin{split}
& \frac{\gamma(d_0d_1d_2)}{ (d_0d_1d_2)^{3/2}}\left(\frac{-d_1d_2}{d_0}\right) \left(\frac{-d_0d_2}{d_1}\right) \left(\frac{-d_0d_1}{d_2}\right) \\
&= \prod_{p \mid d_0d_1d_2} \frac{1}{1 - p^{-2}}
	\prod_{p \mid d_0}\frac{(1-p^{-1})\left(\frac{-d_0d_1}{p}\right)}{p^{3/2}} 
	\prod_{p \mid d_1}\frac{(1-p^{-1})\left(\frac{-d_0d_2}{p}\right)}{p^{3/2}} 
	\prod_{p \mid d_2}\frac{(1-p^{-1})\left(\frac{-d_1d_2}{p}\right)}{p^{3/2}}
\end{split}	
\end{align}
which we shall show is in accordance with \eqref{conj:sum_three_squareful} once taking into account a factor of $\zeta(2)$.
For the sum over $\mathbf{y}$, quadratic reciprocity allows us to write $\left(\frac{y_1y_2}{d_0}\right) \left(\frac{y_0y_2}{d_1}\right) \left(\frac{y_0y_1}{d_2}\right) = \left(\frac{d_1d_2}{y_0}\right) \left(\frac{d_0d_2}{y_1}\right) \left(\frac{d_0d_1}{y_2}\right)$ up to possible $2$-adic factors, which we shall ignore. Expanding out into an Euler product therefore yields
$$\prod_{p \nmid d_0d_1d_2}\left( 1 + \frac{\left(\frac{d_0d_1}{p}\right) + \left(\frac{d_0d_2}{p}\right) + \left(\frac{d_1d_2}{p}\right)}{(1 + p^{-1})p^{3/2}}\right).$$
Multiplying this by $\prod_{p \nmid d_0d_1d_2}(1-p^{-2})$ gives the required Euler factors in \eqref{conj:sum_three_squareful}. This multiplies the expression by $\zeta(2) = \pi^2/6$. This gives a factor of $\pi$ on the numerator, which should be interpreted as a real density similarly to \cite[(3.11)]{Shute21b}.

\subsection*{Conflict of interest}
On behalf of all authors, the corresponding author states that there is no conflict of interest. 

\subsection*{Data availability statement}
There is no associated data.

\bibliographystyle{alpha}
\bibliography{myfile2}

\end{document}